\documentclass{amsart}
\usepackage[utf8]{inputenc}

\usepackage{colonequals}
\usepackage{amsaddr}

\usepackage{float}
\usepackage{tikz-cd}
\usepackage{amsmath}%
\usepackage{amsfonts}%
\usepackage{amsthm}
\usepackage{amssymb}%
\usepackage{graphicx}

\usepackage{hyperref}
\usepackage[all,cmtip]{xy}
\usepackage{afterpage}
\usepackage{geometry}
 \usepackage{comment}
\usepackage{csquotes}
\usepackage{stmaryrd}
\usepackage{lscape} 
\usepackage{pdflscape} 
\usepackage{rotating}
\usepackage{tabularx}
\usepackage{xcolor}

\theoremstyle{plain}
\newtheorem{theorem}{Theorem}[section]
\newtheorem{lemma}[theorem]{Lemma}
\newtheorem{corollary}[theorem]{Corollary}

\newtheorem{proposition}[theorem]{Proposition}

\newcommand{\sqcommdiag}[8]{\[
\xymatrixcolsep{4pc}\xymatrix{
#1 \ar[r]^{#6} \ar[d]_{#5} & #2 \ar[d]^{#7} \\
 #3 \ar[r]_{#8}& #4
}
\]
}

\theoremstyle{definition}
\newtheorem{example}[theorem]{Example}
\newtheorem{definition}[theorem]{Definition}

\theoremstyle{remark}
\newtheorem{remark}[theorem]{Remark}

\renewcommand{\mod}{\textrm{ mod }}

\newcommand{\Z}{\mathbb{Z}}
\newcommand{\Q}{\mathbb{Q}}

\newcommand{\ord}{\mathrm{ord}}
\newcommand{\om}{\omega}

\renewcommand{\P}{\mathbb{P}}
\newcommand{\F}{\mathbb{F}}
\newcommand{\Fp}{\mathbb{F}_p}

\newcommand{\al}{\alpha}

\renewcommand{\O}{\mathcal{O}}

\newcommand{\rr}{\mathfrak{r}}
\newcommand{\dd}{\mathrm{d}}

\newcommand{\pp}{\mathfrak{p}}

\newcommand{\mm}{\mathfrak{m}}
\newcommand{\qq}{\mathfrak{q}}

\newcommand{\Gal}{\text{Gal}}

\newcommand{\hot}{\text{higher-order terms}}
\DeclareMathOperator{\Aut}{{\rm Aut}}

\usepackage{xcolor}

\usepackage{xpatch}

\makeatletter
\xpatchcmd{%
\@maketitle}{%
\ifx\@empty\authors \else \@setauthors \fi
}{%
  \ifx\@empty\authors \else \@setauthors \fi
  \ifx\@empty\addresses \else\@setaddresses\fi
}{\typeout{Patch successful}}{\typeout{Patch failed}}
\makeatother

\begin{document}
\title[Points on modular curves using generalised symmetric Chabauty]{Cubic and quartic points on modular curves using generalised symmetric Chabauty}

 \author{
 Josha Box$^1$, Stevan Gajovi\'{c}$^2$ and Pip Goodman$^3$}
 
 \address{$^1$University of Warwick, $^2$University of Groningen and $^3$University of Bristol}
  
\email{Correspondence to be sent to p.a.goodman@bristol.ac.uk}

\maketitle
\begin{abstract}
     Answering a question of Zureick-Brown, we determine the cubic points on the modular curves $X_0(N)$ for $N \in \{53,57,61,65,67,73\}$ as well as the quartic points on $X_0(65)$.
     To do so, we develop a ``partially relative'' symmetric Chabauty method. Our results generalise current symmetric Chabauty theorems, and improve upon them by lowering the involved prime bound.
     For our curves a number of novelties occur. We prove a ``higher order'' Chabauty theorem to deal with these cases.
     Finally, to study the isolated quartic points on $X_0(65)$, we rigorously compute the full rational Mordell--Weil group of its Jacobian.
\end{abstract}
\section{Introduction}
This work started at the 2020 Arizona Winter School with a question from David Zureick--Brown: Is it possible to determine the finitely many cubic points on $X_0(65)$ despite the infinitude of quadratic points? In this article we answer this question affirmatively by developing a ``partially relative'' symmetric Chabauty method (Theorem \ref{bigtheorem}).
This theorem generalises the work of Siksek \cite{siksek} on symmetric Chabauty, and moreover can be used with a larger set of small primes. In certain cases this ``first order'' method still fails. To overcome this, we develop a Chabauty method (Theorem \ref{biggertheorem}) that takes into account higher order terms of the relevant expansions of differential forms.

The case of $X_0(65)$ does not stand alone. Often for higher degrees $d>2$, it happens that a curve admits infinitely many points of lower degree $e<d$, while the set of degree $d$ points is finite. This is the first Chabauty method which has the potential to compute all degree $d$ points in some of such cases, provided the rank of the Mordell--Weil group of the curve's Jacobian is not too large. This opens the way for the provable determination of all points of fixed degree $d$ (notably when $d>2$) on various classes of interesting curves. We illustrate this with the following result.

\begin{theorem}\label{thm:cubic+quartic}
The set of cubic points on each of the curves
\[
X_0(53),\;\; X_0(57),\;\; X_0(61), \;\;X_0(65),\;\; X_0(67) \text{ and }X_0(73)
\]
is finite and listed in Section \ref{section_results}.
The quartic points on $X_0(65)$ form an infinite set.
This infinite set consists of inverse images of quadratic points on the quotient curve $X^+_0(65)$ and a finite number of isolated points.
The isolated points are listed in \S \ref{subsection_quarticpoints}.
\end{theorem}

For a curve $X$ over a number field $K$, denote by $X^{(d)}$ its $d$th symmetric power. The $K$-rational points on $X^{(d)}$ are exactly the $K$-rational effective degree $d$ divisors on $X$, and the $\mathrm{Gal}(\overline{K}/K)$-orbit of each point $P\in X$ defined over a degree $d$ extension  of $K$ gives rise to such a divisor. By studying the $K$-rational points on $X^{(d)}$ we can thus study the $L$-rational points on $X$ for all degree $d$ extensions $L/K$ simultaneously.

An important tool to study $X^{(d)}$ is the Abel--Jacobi map. Given a $K$-rational degree $d$ divisor $D_0$ on $X$, we define it as
\[
\iota\colon\; X^{(d)}\to J(X),\; D\mapsto [D-D_0].
\]
So when does $X/K$ have infinitely many points of degree $d$ over $K$? Certainly when there is a degree $d$ map $\rho\colon X\to \P^1$, as the inverse image of $\P^1(K)$ provides such an infinite set. More generally, when there is a map $\rho\colon X\to C$ of degree $e$ to a curve $C$ such that $C^{(f)}(K)$ is infinite, $e\cdot f\leq d$ and there exists $P\in X^{(d-ef)}(K)$: then such an infinite set is 
\[
P+\rho^*C^{(f)}(K) \subset X^{(d)}(K),
\]
and $\iota$ maps this set into a translate of the abelian subvariety $\rho^*J(C)$ of $J(X)$.

Define $P\in X^{(d)}(K)$ to be an \emph{isolated point} when it is neither in the inverse image of $\P^1(K)$ under a degree $d$ map $X\to \P^1$, nor does $\iota(P)$ lie in a translate of a positive rank abelian subvariety of $J(X)$ contained in $\iota(X^{(d)})$. Recently, Bourdon, Ejder, Liu, Odumodu and Viray \cite[Theorem 4.2]{belov} have shown that 
\begin{itemize}
    \item[(i)] $X$ has infinitely many points of degree $d$ over $K$ if and only if $X^{(d)}(K)$ contains a non-isolated point, and 
    \item[(ii)] there are finitely many isolated points in $X^{(d)}(K)$.
\end{itemize}
This result provides a road map for studying $X^{(d)}(K)$: first describe each infinite set, then determine the finite set of isolated points. This is what we have done for the modular curves listed above, using a generalisation of Chabauty's method.

Chabauty's method is, classically, a method for effectively computing the rational points on a curve in the special case when the rank $r$ of its Mordell--Weil group is strictly smaller than its genus $g$. The effectiveness is due to Coleman \cite{coleman}, who realised that Chabauty's finiteness proof \cite{chabauty} could be made effective using the machinery of locally analytic $p$-adic functions. 

While nowadays many focus on weakening the $r<g$ condition in Chabauty's method via non-abelian generalisations (see e.g. \cite{kim1} and \cite{bbbmtv}), the reach of Chabauty's method is still being increased.  After partial results of Klassen \cite{klassen}, Siksek \cite{siksek} extended Coleman's ideas to obtain an effective method for computing the $K$-rational points (when finite) on the $d$th symmetric powers of curves over any number field $K$, provided the stronger condition
\[
r<g-(d-1)
\]
is satisfied. A similar result was obtained by Derickx, Kamienny, Stein and Stoll \cite{dkss} using the theory of formal immersions, which they applied to points of degree up to 7 on several modular curves. 

Their method has led to several important results: it was used by Derickx, Najman and Siksek to prove that all elliptic curves over totally real cubic fields are modular \cite{derickx}, and helped solve the puzzle of classifying the finite groups that appear as the torsion subgroup of the Mordell--Weil group of an elliptic curve over a cubic field \cite{dzb}. In the latter article, the Chabauty-like method for symmetric powers was used to determine $X_1(22)^{(3)}(\Q)$, $X_1(25)^{(3)}(\Q)$ and the image of $X_1(65)^{(3)}(\Q)$ in $X_0(65)^{(3)}(\Q)$. Using our method, we can now describe the complete set $X_0(65)^{(3)}(\Q)$ and list the finitely many cubic points on $X_0(65)$. 

Furthermore, Siksek \cite{siksek} developed a relative version of his symmetric Chabauty method, which can be used to determine the isolated points on $X^{(d)}$ if the infinite set consists entirely of pullbacks. To be precise, given a map $\rho\colon X\to C$ of curves over $K$ of degree $d$, Siksek's relative Chabauty method can determine the remainder $X^{(d)}(K)\setminus \rho^*C(K)$ if that consists entirely of isolated points. This method was employed by the first named author \cite{box} to describe the quadratic points on $X_0(N)$ for $N\in \{43,53,57,61,65,67,73\}$.  

The curve $X_0(65)$ admits a degree 2 map $\rho\colon X_0(65)\to X_0(65)/\langle w_{65} \rangle =\colon X^+_0(65)$, where $w_{65}$ is the Atkin--Lehner involution. This quotient $X_0^+(65)$ is an elliptic curve of rank 1, and $\rho^*(X_0^+(65)(\Q))$ is an infinite set of degree 2 points.
In \cite{box}, Siksek's relative Chabauty method was used to show that the only isolated degree 2 points are sums of two cusps. For degree 3, however, we obtain for each cusp $c\in X_0(65)(\Q)$ such an infinite set 
\[
c+\rho^*(X_0^+(65)(\Q)) \subset X_0(65)^{(3)}(\Q),
\]
and Siksek's method cannot be applied anymore to study the isolated points of degree 3, as the infinite sets are not pullbacks.

It is exactly this problem that we solved, by generalising Siksek's ideas to obtain a ``partially relative symmetric Chabauty method'' that has the potential to determine the isolated degree $d$ points on a curve $X$ if the infinite sets are of the form
\[
P+\rho_1^*C_1^{(\ell_1)}(K)+\cdots+\rho_n^*C_n^{(\ell_n)}(K),
\]
where $P\in X^{(e)}(K)$ and $\rho_i\; \colon X\to C_i$ are maps of degree $d_i$ such that $e+\ell_1d_1+\cdots+\ell_nd_n=d$. The result that makes this possible is Theorem \ref{bigtheorem}. This theorem has already contributed in \cite{box2} to the proof that all elliptic curves over quartic fields not containing $\sqrt{5}$ are modular.

Our work on the modular curves $X_0(N)$ for $N\in \{53,57,61,65,67,73\}$ extends a string of papers studying the modular curves $X_0(N)$ of genus $g\in \{2,\ldots,5\}$. Bruin and Najman \cite{bruin} determined the isolated quadratic points on the hyperelliptic $X_0(N)$ of these genera with finite Mordell--Weil group.  Subsequently, \"Ozman and Siksek \cite{ozman} determined the finitely many quadratic points on the non-hyperelliptic $X_0(N)$ of genus $g\in \{2,\ldots,5\}$ with finite Mordell--Weil group, and the first named author \cite{box} described the quadratic points on those with infinite Mordell--Weil group. 

The curves for which we determine the cubic points are exactly those $X_0(N)$ of genus $g\in \{2,\ldots,5\}$ that have infinite Mordell--Weil group and for which the Chabauty condition on the rank holds true. Only two of these curves also satisfy the rank condition for quartic points: $X_0(65)$ and $X_0(57)$.
However, a genus 5 curve admits infinitely many degree four maps to $\P^1$ (see \S \ref{subsection_Qcurves}).
This prevents us from determining the quartic points on $X_0(57)$.
For $X_0(65)$ on the other hand, we are in luck.
All degree four maps to $\P^1$ defined over $\Q$ factor through the elliptic curve $X_0^+(65)$.
Applying our partial relative Chabauty method with respect to this quotient, we determine all of the isolated quartic points on $X_0(65)$.
The \texttt{Magma} \cite{magma} code to verify all computations made in this paper can be found at
\[
\texttt{\href{https://github.com/joshabox/cubicpoints/}{https://github.com/joshabox/cubicpoints/}} \;.
\]

\section{A general symmetric Chabauty theorem}
In this section we present a common generalisation of Theorems 3.2 and 4.3 in \cite{siksek}, which are Chabauty-type theorems for computing  $K$-rational points on symmetric powers of curves. 

Before this, we give an overview of Siksek's Chabauty method. 
\subsection{Uniformisers and differentials}\label{prelimsection}
Let $K$ be a finite extension of $\Q_p$ with ring of integers $R$ and residue field $k$. Consider a curve $X/K$ together with a minimal proper regular model $\mathcal{X}/R$ for $X$. We write $\widetilde{X}$ for the special fibre of $\mathcal{X}$, and similarly denote reductions of objects associated to $\mathcal{X}$ with a tilde. Denote by $\O_S$ the sheaf of functions on a scheme $S$. For a sheaf $\mathcal{F}$ on $S$, we denote by $\mathcal{F}_s$ its stalk at $s$. When $A$ is a discrete valuation ring (DVR), we denote by $\widehat{A}$ its completion. Finally, when $s\in S$ is such that $\O_{S,s}$ is a DVR,  $U\subset S$ is an open subset containing $s$ and $\mathcal{F}$ is an $\O_S$-module, we denote by $\mathrm{loc}_s$ the map $\mathcal{F}(U)\to \mathcal{F}_s$. For $x\in X^{(d)}(K)$ and $y\in \widetilde{X}^{(d)}(k)$, we denote by $D(x)$ and $D(y)$ the points in $X^{(d)}(K)$ reducing to $\widetilde{x}$ and $y$ respectively.

When $x\in X$ is non-singular, a uniformiser $t \in \widehat{\O}_{X,x}$ is called a \emph{local coordinate}  at $x\in X$. When $t$ also reduces to a uniformiser in the reduction $\widehat{\O}_{\widetilde{X},\widetilde{x}}$, we say that it is a \emph{well-behaved local coordinate} or \emph{well-behaved uniformiser}. This just means that the maximal ideal of $\widehat{\O}_{\mathcal{X},\widetilde{x}}$ is generated by $t$ and $p$ as an $R$-module. The following facts can be found for example in \cite{lorenzinitucker}. When $t$ is a well-behaved local coordinate at $x$, we can evaluate $t$ at points in the residue disc $D(x)$ of $x$, yielding a bijection between $D(x)$ and the maximal ideal of $R$.

Denote by $\Omega_{X/K}$ and $\Omega_{\mathcal{X}/R}$ (sometimes abbreviated to $\Omega$) the sheaves of regular differentials on $X$ and $\mathcal{X}$ respectively. A choice of uniformiser $t\in \widehat{\O}_{X,x}$ gives rise to the identifications $\widehat{\O}_{X,x}=K\llbracket t \rrbracket$ and $\widehat{\Omega}_{X,x}=K\llbracket t\rrbracket \dd t$. If $t$ is moreover well-behaved, we have $\widehat{\O}_{\mathcal{X},\widetilde{x}}=R\llbracket t \rrbracket$ and $\widehat{\Omega}_{\mathcal{X},\widetilde{x}}=R\llbracket t\rrbracket \dd t$. 

Now consider $\omega\in H^0(X,\Omega_{X/K})$. Since $H^0(\mathcal{X},\Omega_{\mathcal{X}/R})$ is a lattice in $H^0(X,\Omega_{X/K})$, after multiplication by a constant in $R$, we may assume that $\omega\in H^0(\mathcal{X},\Omega_{\mathcal{X}/R})$. In sum, given a well-behaved uniformiser $s$ at a point $Q\in X(K)$, we can write
\begin{align}
\label{omegaext}
\mathrm{loc}_Q(\om)=\sum_{n=0}^{\infty} a_ns^n \dd s \text{ with } a_n\in R \text{ for all } n.
\end{align}
Consider any point $P_0\in X(\overline{K})$. 
\begin{lemma}
\label{diffisolemma}
The map $\iota\colon X\to J(X),\; P\mapsto [P-P_0]$ induces an isomorphism
\[
\iota^*\colon \; H^0(J,\Omega_{J/K})\simeq H^0(X,\Omega_{X/K})
\]
of global differential forms independent of the choice of $P_0$. 
\end{lemma}
\begin{proof}
See e.g. \cite[Proposition 2.1]{siksek}.
\end{proof}
We shall thus use $\iota^*$ to pass between these two spaces. 

Now suppose that we have a map $\rho\colon X\to C$ between two curves over $K$. We obtain a pushforward map $\rho_*\colon \; J(X)\to J(C)$ and pullback map $\rho^*\colon \; J(C)\to J(X)$, leading to a decomposition up to isogeny
\[
J(X)\sim J(C)\times A,
\]
where $A\subset J(X)$ is an abelian subvariety. Denote by $\pi_A$ the map $J(X)\to A$. We also obtain a push-forward, or ``trace map'' on meromorphic differentials
\[
\mathrm{Tr}\colon\; \Omega_{K(X)/K}\to \Omega_{K(C)/K}.
\]
We consider $H^0(X,\Omega_{X/K})$ as a subspace of $\Omega_{K(X)/K}$, and similarly for $C$. 
\begin{lemma}
\label{difflemma}
We have
\begin{itemize}
    \item[(i)] $\Omega_{K(X)/K}=\rho^*\Omega_{K(C)/K}\oplus \mathrm{Ker}(\mathrm{Tr})$, and
\item[(ii)]$
\iota^*\pi_A^*H^0(A,\Omega_{A/K})=\mathrm{Ker}(\mathrm{Tr})\cap H^0(X,\Omega_{X/K})$.
\end{itemize}
\end{lemma}
\begin{proof}
Part (i) follows from surjectivity of the trace map. For part (ii) it suffices to compare dimensions and to check that $\iota^*\pi_A^*H^0(A,\Omega_{A/K})\subset\mathrm{Ker}(\mathrm{Tr})$.
\end{proof}

\subsection{Coleman integration}
We consider the notation from the previous subsection. Again let $K$ be a finite extension of $\Q_p$. Also let $B/K$ be any abelian variety with good reduction at all primes in $K$ above $p$. In \cite[Section II]{coleman2}, Coleman defines a pairing, now called \emph{Coleman integration}:
\begin{align}
\label{integrationpairing}
H^0(B,\Omega_{B/K})\times B(K)\to K,\;\; (\om,P)\mapsto \int_P \om.
\end{align}
We note that Coleman defines this integration in a much more general setting than this, but we shall only be concerned with the above case of abelian varieties and differentials of the first kind. 
\begin{proposition}
\label{intprops1}
The integration pairing (\ref{integrationpairing}) is
\begin{itemize}
    \item[(i)] locally analytic in $P\in B(K)$,
    \item[(ii)] $\Z$-linear on the right,
    \item[(iii)] $K$-linear on the left,
    \item[(iv)] its left-hand kernel is zero, 
    \item[(v)] its right-hand kernel is $B(K)_{\mathrm{tors}}$, and
    \item[(vi)] if $g\colon B\to B'$ is a morphism of abelian varieties over $K$, then
    \[
    \int_Pg^*\om = \int_{g(P)} \om
    \]
    for all $\om \in H^0(B',\Omega_{B'/K})$ and $P\in B(K)$. 
\end{itemize}
\end{proposition}
\begin{proof}
Parts (i) and (ii) are \cite[Theorem 2.8]{coleman2}. Part (iii) follows from part (i) and \cite[Proposition 2.4 (i), (iii)]{coleman2}. Part (iv) follows from \cite[Theorem 2.8 (ii)]{coleman2}, and Part (v) is \cite[Theorem 2.11]{coleman2}. Finally, part (vi) is \cite[Theorem 2.7]{coleman2}.
\end{proof}

\begin{lemma}
\label{ranklemma}
Suppose that $G\subset B(K)$ is a subgroup of rank $r$. Then there is a vector space $\mathcal{V}\subset H^0(B,\Omega_{B/K})$ of dimension at least $\mathrm{dim}(B)-r$ such that for all $\om\in \mathcal{V}$, we have
\[
\int_{P}\om=0 \text{ for all } P\in \overline{G},
\]
where the closure is inside the $p$-adic topology on $B(K)$.
\end{lemma}
\begin{proof}
Suppose that $D_1,\ldots,D_r$ generate $G$ up to torsion. Then $\overline{G}\otimes K=KD_1+\cdots+KD_r$, of dimension $\leq r$. Now we find $\mathcal{V}$ because (\ref{integrationpairing}) extends to an exact pairing between $H^0(B,\Omega_{B/K})$ and $B(K)\otimes K$ by Proposition \ref{intprops1}.
\end{proof}
When $B=J(X)/K$ and $G=J(X)(K)$, we call this space $\mathcal{V}$ the \emph{space of annihilating differentials}. Moreover, when we have a map $\rho\colon X\to C$ as in the previous section, we call $\mathcal{V}\cap \mathrm{Ker}(\mathrm{Tr})$ the \emph{space of annihilating differentials with trace zero}. This is the pullback along $\pi_A$ of the space of annihilating differentials on $A$, where $A$ is such that $J(X)\sim J(C)\times A$.

More notation: by pulling back along $\iota$, we obtain integrals
\[
H^0(X,\Omega_{X/K})\times X(K)^2\to K,\;\;(\om,Q,P)\mapsto \int_Q^P\om :=\int_{[P-Q]}(\iota^*)^{-1}(\om).
\]

\begin{proposition}
\label{intprops2}
\begin{itemize}\item[(i)]For $P,Q\in X(K)$ such that $P\in D(Q)$, and a well-behaved uniformiser $s$ at $Q$, we have for each $\om\in H^0(X,\Omega_{X/K})$
\[
\int_{Q}^P\om = \sum_{n=0}^{\infty}\frac{a_n}{n+1}s(P)^{n+1},
\]
where $\mathrm{loc}_Q(\om)=\sum_n a_n s^n$ as in (\ref{omegaext}). We call such an integral between points in the same residue class a \emph{tiny integral}.
\item[(ii)] If $\rho\colon X\to C$ is a non-constant morphism of curves over $K$ with good reduction, then
\[
\int_{\rho^*D}\om =\int_D\mathrm{Tr}(\om)
\]
for all $\om \in H^0(X,\Omega_{X/K})$ and every degree 0 divisor $D$ on $C$ defined over $K$. 
\end{itemize}
\end{proposition}
\begin{proof}
Part (i) follows from Proposition \ref{intprops1} (i) together with the Fundamental Theorem of Calculus proved in \cite[Proposition 2.4 (ii)]{coleman2}. Part (ii) is \cite[Lemma 2.2]{siksek}.
\end{proof}
\subsection{An overview of the symmetric Chabauty--Coleman method}\label{overviewsection}
We refer the reader to \cite{wetherell} and \cite{poonen} for a clear overview of the Chabauty--Coleman method in the classical case. Here we give an overview of Chabauty and Coleman's original ideas in the symmetric power setting, after which we explain what needs to be changed in the relative case. For simplicity we work over $\Q$, but  everything generalises to number fields.

We consider an integer $d$, a prime $p$ and a curve $X/\Q$ of genus $g_X\geq 2$ and minimal proper regular model $\mathcal{X}/\Z_p$ for $X_{\Q_p}$. To determine $X^{(d)}(\Q)$, it suffices to determine the rational points contained in each of the residue discs separately.  So consider $\widetilde{\mathcal{Q}}\in X^{(d)}(\F_p)$ and its inverse image under the reduction map, the residue disc $D(\widetilde{\mathcal{Q}})\subset X^{(d)}(\Q_p)$. Assume there exists $D_0\in X^{(d)}(\Q)$. We use it to define the Abel--Jacobi map:
\[
\iota\colon X^{(d)}\to J(X),\;\; D\mapsto [D-D_0].
\]
Chabauty's idea was to consider the diagram
\begin{equation}
\label{diag}
\begin{tikzcd}
D(\widetilde{\mathcal{Q}})\cap X^{(d)}(\Q) \arrow[rr,"\iota"] \arrow[d] &  & J(X)(\Q) \arrow[d] \\
D(\widetilde{\mathcal{Q}}) \arrow[rr,"\iota"]         &  & J(X)(\Q_p)    
\end{tikzcd}
\end{equation}
and instead determine 
\begin{align}
\label{int}
\iota(D(\widetilde{\mathcal{Q}}))\cap \overline{J(X)(\Q)},
\end{align}
where the closure is inside the $p$-adic topology on $J(X)(\Q_p)$. This set contains $\iota(D(\widetilde{\mathcal{Q}})\cap X^{(d)}(\Q))$. Let $r$ be the rank of $J(X)(\Q)$. By Lemma \ref{ranklemma}, we obtain a space $\mathcal{V}\subset H^0(X,\Omega_{X/K})$ of dimension $\mathrm{dim}(\mathcal{V})\geq g_X-r$ such that for all $\om\in \mathcal{V}$, we have
\[
\int_D\om=0 \text{ for all } D\in \overline{J(X)(\Q)}.
\]
In particular, when the \emph{Chabauty condition}
\begin{align}
\label{chabautycond}
r<g_X-(d-1)
\end{align}

is satisfied, there are $d$ such linearly independent differentials, and dimensions suggest that
\[
\mathcal{Z}_p(\widetilde{\mathcal{Q}})\colon=\left\{\mathcal{P}\in D(\widetilde{\mathcal{Q}}) \mid \int_{\iota(\mathcal{P})}\om =0 \text{ for all }\om\in \mathcal{V}\right\}
\]
is finite.  In the case $d=1$, Chabauty \cite{chabauty} proved that $\overline{J(\Q)}\cap \iota(D(\widetilde{\mathcal{Q}}))$ is indeed finite when (\ref{chabautycond}) is satisfied. For $d>1$, $X^{(d)}(\Q)$ can still be infinite even when (\ref{chabautycond}) holds, e.g. due to the existence of a map $\rho\colon X\to C$ of degree $\leq d$. For now, however, let us assume that $X$ satisfies (\ref{chabautycond}), and that we have no reason to believe that $X^{(d)}(\Q)\cap D(\widetilde{\mathcal{Q}})$ is infinite regardless. 

It was Coleman's idea \cite{coleman} to introduce integration of differentials and compute $\mathcal{Z}_p(\widetilde{\mathcal{Q}})$ instead. Such zero sets can be computed for $d=1$ by evaluating  Coleman integrals between $\Q_p$-rational points.
In \texttt{Sage} on hyperelliptic curves, the computation of Coleman integrals over $\Q_p$ has been implemented  by Balakrishnan, Bradshaw, and Kedlaya \cite{BBK-OddHC-ColemanIntegration}, \cite{Balakrishnan-EvenHC-ColemanIntegration}, while Balakrishnan and Tuitman \cite{BalakrishnanTuitman} wrote a \texttt{Magma} implementation for plane curves. For $d>1$, however, one would need to evaluate Coleman integrals between points $P,Q\in X(K)$ for extensions $K/\Q_p$ of degree $>1$; this has only been done for superelliptic curves over unramified extensions of $\Q_p$, by recent work of Best \cite{Alex-ColemanIntegration-Unramified-Superelliptic}.

Instead, therefore, we shall restrict our attention to tiny (i.e. directly computable) integrals only: this suffices when combining the information from multiple primes $p$ using the Mordell--Weil sieve (see Section \ref{section_MWsieve}). The drawback here is that information on $J(X)(\Q)$ is needed. Thanks to the sieve, we need to consider only residue discs $D(\widetilde{\mathcal{Q}})$ containing a known point $\mathcal{Q}\in X^{(d)}(\Q)$. For each such point $\mathcal{Q}$, (assuming that $X^{(d)}(\Q)$ is finite) there is a prime $p$ such that $X^{(d)}(\Q)\cap D(\widetilde{\mathcal{Q}})=\{\mathcal{Q}\}$. To compute $X^{(d)}(\Q)$, it then suffices to have a criterion to decide whether $\mathcal{Z}_p(\widetilde{\mathcal{Q}})=\{\mathcal{Q}\}$. 

Given the known point $\mathcal{Q}$, we can choose $D_0=\mathcal{Q}$ to define $\iota$. Then for $\mathcal{P}\in D(\widetilde{\mathcal{Q}})$, the integral
$
\int_{\iota(\mathcal{P})}\om
$
is a sum of $d$ tiny integrals. By studying the power series obtained from these tiny integrals via Proposition \ref{intprops2} (i), Siksek \cite[Theorem 3.2]{siksek} found a criterion for deciding whether  $Z_p(\widetilde{\mathcal{Q}})=\{\mathcal{Q}\}$.

\subsection{An overview of our relative Chabauty--Coleman method}
We continue the notation from the previous subsection. We now assume that we \emph{do} have reason to believe that $X^{(d)}(\Q)\cap D(\widetilde{\mathcal{Q}})$ is infinite, due to the existence of a curve $C$ with minimal proper regular model $\mathcal{C}/\Z_p$,  a map $\rho\colon\mathcal{X}\to \mathcal{C}$ such that $\rho\colon X\to C$ has degree $f$, a positive integer $e$ such that $e\cdot f\leq d$, and known points $P\in X^{(d-ef)}(\Q)$ and $Q\in C^{(e)}(\Q)$ such that $\mathcal{Q}:=P+\rho^*Q\in D(\widetilde{\mathcal{Q}})$. Indeed, we then have a family
\[
P+\rho^*C^{(e)}(\Q)\subset X^{(d)}(\Q)
\]
intersecting $D(\widetilde{\mathcal{Q}})$. Even when $X$ satisfies the Chabauty condition (\ref{chabautycond}), the approach from the previous subsection fails when either $C^{(e)}(\Q)$ is infinite or when $C^{(e)}$ does not satisfy the Chabauty condition $r_C<g_C-(e-1)$ itself, where $r_C=\mathrm{rk}J(C)(\Q)$ and $g_C$ is the genus of $C$.

As in Section \ref{prelimsection}, there is an abelian variety $A\subset J(X)$ such that $J(X)\sim J(C)\times A$, and we define $\pi_A\colon J(X)\to A$. As before, define the Abel--Jacobi map $\iota$ using $\mathcal{Q}$. We now replace $J(X)$ in Diagram~\eqref{diag} by $A$, and note that
\[
\pi_A\circ\iota \left(P+\rho^*C^{(e)}(\Q)\right)=\{0\},
\]
so the entire family has been collapsed to a single point on $A$. It may be easier to determine  $\overline{A(\Q)}\cap \pi_A(\iota(D(\widetilde{\mathcal{Q}})))$ than it is to determine $\overline{J(X)(\Q)}\cap\iota(D(\widetilde{\mathcal{Q}}))$. Again by Lemma \ref{difflemma}, we find a space $\mathcal{V}\subset H^0(A,\Omega_{A/K})$ of dimension $\mathrm{dim}(\mathcal{V})\geq \mathrm{dim}(A)-\mathrm{rank}(A(\Q))$ such that for all $\om\in \mathcal{V}$
\[
\int_D\om =0 \text{ for each } D\in \overline{A(\Q)}.
\]
Let $r_X$ be the rank of $J(X)(\Q)$, $r_C$ the rank of $J(C)(\Q)$, and $g_X$ and $g_C$ be the genera of $X$ and $C$ respectively. If the Chabauty condition 
\begin{equation}
\label{generalchabcond}
r_X-r_C<g_X-g_C-(d-1)
\end{equation}
is satisfied, then $\mathrm{dim}(\mathcal{V})\geq d$ and the dimensions suggest that the common zero set of $\int\om$ for $\om\in \mathcal{V}$ has finite intersection with $\pi_A\circ \iota(D(\widetilde{\mathcal{Q}}))$. Define 
\[
\mathcal{Z}_{p,A}(\widetilde{\mathcal{Q}}):=\left\{\mathcal{P}\in D(\widetilde{\mathcal{Q}}) \mid \int_\mathcal{Q}^\mathcal{P}\om =0 \text{ for all } \om \in \mathcal{V}\cap \mathrm{Ker}(\mathrm{Tr})\right\},
\]
where $\mathrm{Tr}\colon\; \Omega_{K(X)/K}\to \Omega_{K(C)/K}$ is the trace map. By Lemma \ref{difflemma} (ii) and Proposition \ref{intprops1} (vi), this is the inverse image to $D(\widetilde{\mathcal{Q}})$ of the common zero set of the integrals $\int\om$ for $\om \in \mathcal{V}$. 

Analogous to the case of the previous section, we now desire a criterion to decide if \[
\mathcal{Z}_{p,A}(\widetilde{\mathcal{Q}})=(P+\rho^*C^{(e)}(\Q))\cap D(\widetilde{\mathcal{Q}}).
\]
This is the purpose of Theorem \ref{bigtheorem}.
\subsection{The main theorem}
Consider a point $Q$ on a curve $X$ over a field $K$, and a regular 1-form $\om\in H^0(X,\Omega_{X/K})$. We expand $\om$ around $Q$ in terms of a uniformiser $t_Q$ at $Q$, giving $\mathrm{loc}_Q(\om)=\sum_{j\geq 0} a_jt_Q ^j\dd t_Q$. We define 
\begin{align*}
v(\om,t_Q,k)&:=\left(-a_0, a_1,\ldots, (-1)^{k}a_{k-1}\right).
\end{align*}
When $\mathcal{Q}$ is an effective $K$-rational divisor on $X$, we denote by $K(\mathcal{Q})$ the (Galois) extension obtained by adjoining to $K$ all points in the support of $\mathcal{Q}$. 
\begin{theorem}\label{bigtheorem}
 Let $\rho_j\colon X\rightarrow C_j$ for $j\in \{1,\ldots,h\}$ be degree $d_j$ maps of curves over a number field $K$, and consider given an effective divisor
 \[
 \mathcal{Q}= \mathcal{Q}_0+ \mathcal{Q}_1+\cdots+\mathcal{Q}_h, \text{ where } \mathcal{Q}_0\in X^{(e)}(K) \text{ and } \mathcal{Q}_j\in \rho_j^*C_j^{(\ell_j)}(K) \text{ for } j\geq 1.
 \]
 Let $n=e+d_1\ell_1+\cdots+d_h\ell_h$. Suppose that $\rr$ is a prime in $\O_K$ of good reduction for $X$ and each $C_j$.  Let $p$ be the rational prime contained in $\rr$. 
 \begin{itemize}
     \item[(1)] Assume that the supports of $\mathcal{Q}_1,\ldots,\mathcal{Q}_h$ are pairwise disjoint, and no point in the support of any $\mathcal{Q}_i$ for $i\geq 1$ has ramification degree under $\rho_i$ divisible by $p$. 
     \item[(2)] Let $N$ be the maximum of the ramification indices of $p$ in $K(Q_i,Q_j)$ for $i,j\in \{1,\ldots,k\}$. Assume that $p\geq N+2$. 
 \end{itemize}
 Write $\lambda=\ell_1+\cdots+\ell_h$. Let $\mathcal{V}_0$ be the space of annihilating differentials on $X$ with trace zero with respect to each $\rho_j$, and consider a basis $\widetilde{\omega}_1,\ldots,\widetilde{\omega}_q$ for the image of $\mathcal{V}_0\cap H^0(\mathcal{X},\Omega)$ under the mod $\rr$ reduction map on differentials. 
 Let $\pp$ be a prime in $K(\mathcal{Q})$ above $\rr$, and denote reductions of points with respect to $\pp$ with a tilde.  Write $\mathcal{Q}=n_1Q_1+\cdots+n_kQ_k$ with $Q_1,\ldots,Q_k\in X$ distinct points and each $n_i\geq 1$, and let $t_{\widetilde{Q}_i}$ be a uniformiser at $\widetilde{Q}_i$ for each $i$. 
 \begin{itemize}
     \item[(3)] Assume that the matrix \[
\widetilde{\mathcal{A}}:=\begin{pmatrix}
v(\widetilde{\om}_1,t_{\widetilde{Q}_1},n_1) & v(\widetilde{\om}_1,t_{\widetilde{Q}_2},n_2) & \cdots & v(\widetilde{\om}_1,t_{\widetilde{Q}_k},n_k) \\
v(\widetilde{\om}_2,t_{\widetilde{Q}_1},n_1) & v(\widetilde{\om}_2,t_{\widetilde{Q}_2},n_2) & \cdots & v(\widetilde{\om}_2,t_{\widetilde{Q}_k},n_k) \\
\vdots & \ddots & \ddots & \vdots  \\
v(\widetilde{\om}_q,t_{\widetilde{Q}_1},n_1) & v(\widetilde{\om}_q,t_{\widetilde{Q}_2},n_2) & \cdots & v(\widetilde{\om}_q,t_{\widetilde{Q}_k},n_k) 
\end{pmatrix}  
\]
has rank $n-\lambda$. 
\end{itemize}
Then every $\mathcal{P}\in X^{(n)}(K)$ in the mod $\rr$ residue disc of $\mathcal{Q}$ is in fact contained in  \[
\mathcal{Q}_0+\rho_1^*C_1^{(\ell_1)}(\overline{K})+\cdots+\rho_h^*C_h^{(\ell_h)}(\overline{K}).
\]
\end{theorem}
We postpone the proof of this theorem to Section \ref{proofsec}. 
\begin{remark}
By Proposition \ref{coefficients conditions} in the next section, $n-\lambda$ is in fact the maximal possible rank of this matrix. Moreover, note that condition (3) is not satisfied when two distinct points $Q_i\neq Q_j$ have the same reduction mod $\pp$. If they do, at least two columns in $\widetilde{\mathcal{A}}$ would agree, reducing its rank to less than $n-\lambda$ due to the nature of the linear equations between the columns obtained from Proposition \ref{coefficients conditions}. 
\end{remark}
\begin{remark}
 Our lower bound $p\geq N+2$ leads to a general lower bound $p\geq n(n-1)+2$.  In particular, for $n=2$, $p\geq 5$ suffices; for $n=3$, $p\geq 11$ suffices; and for $n=4$, $p\geq 17$ suffices. 
\end{remark}
\begin{remark}
This theorem generalises \cite[Theorems 3.2 and 4.3]{siksek}. More precisely, the case $\mathcal{Q}=\mathcal{Q}_0$ corresponds to \cite[Theorem 3.2]{siksek} and the case $\mathcal{Q}=\mathcal{Q}_1$ corresponds to \cite[Theorem 4.3]{siksek}. In both cases, we significantly improve the lower bound on the prime $p$ compared to \cite{siksek}.  Part of the improvement is due to the introduction in the proof of elementary symmetric polynomials replacing power sums. This removes denominators in the power series expansion and hence in the matrix $\widetilde{\mathcal{A}}$; this is an idea also used in \cite{dkss}. Another improvement comes from a Galois theory argument, exploiting the fact that $\mathcal{P}$ is $K$-rational. 
\end{remark}

\begin{remark}
When $p$ is small, it sometimes happens that $X^{(n)}(\Q)$ surjects onto $X^{(n)}(\F_p)$, in which case this theorem may be used to determine $X^{(n)}(\Q)$ directly (i.e. without sieving). This has the advantage of requiring no information on the generators of the Mordell--Weil group of $J(X)$.  When $\mathcal{Q}=\mathcal{Q}_0$ and the annihilating differentials come from a rank zero quotient of $J(X)$, it is best to use the formal immersion criterion of Derickx, Kamienny, Stein and Stoll \cite[Chapter 3 Proposition 3.7]{derickxthesis}  instead, which gives the same statement but works for all primes $p\geq 2$ of good reduction. See for example \cite[\S6.1]{dzb} where this is done for $X_1(22)^{(3)}(\Q)$ and $X_1(25)^{(3)}(\Q)$ with $p=3$. Our theorem covers all relative and positive rank cases, and can often be used with $p=3$.
\end{remark}
\begin{remark}
While we have stated the theorem for multiple maps $\rho_1,\ldots,\rho_h$, we only need the case $h=1$ in our examples, in which case $\mathcal{P}\in \mathcal{Q}_0+\rho_1^*C_1^{(\ell_1)}(K)$ (the pull-back part is $K$-rational since $\mathcal{P}$ and $\mathcal{Q}_0$ are). Also note that condition (1) is quite limiting when $h>1$. If, for example, there are two degree 2 maps $\rho_1\colon X\to C_1$ and $\rho_2\colon X\to C_2$ and $Q\in X(\Q)$ then $\rho_1^*\rho_1(Q)+\rho_2^*\rho_2(Q)\in X^{(4)}(\Q)$ does not satisfy (1). 
\end{remark}
\subsection{Trace maps and ramification}\label{Trace maps} We note that all statements in this section are essentially well-known results in algebraic number theory. We nonetheless give proofs because we could not find the exact statements in the literature.

Any map $\rho\colon \; X\to C$ of curves  as in the statement of Theorem \ref{bigtheorem}  may be ramified at certain points of degree at most $d$. 
\begin{example}
\label{example:67_ramified_point}
When $\rho$ is the quotient map $\rho\colon\; X_0(67)\to X_0^+(67)$, there is a non-cuspidal rational point $Q\in X_0(67)(\Q)$ that ramifies. Ramifying here means that $w_{67}(Q)=Q$. In order to deal with degree 3 effective divisors such as $3Q=Q+\rho^*\rho(Q)\in X_0(67)^{(3)}(\Q)$, we study in more detail how differentials transform under (ramified) maps of curves.
\end{example}
First, we briefly recall some facts about maps between Krull domains. We consider a integral extension $A\to B$ of Krull domains. Krull domains can be viewed as higher-dimensional generalisations of Dedekind domains; see e.g.\ \cite{matsumura} for their definition and theory. Consider a minimal prime ideal $\pp\subset A$. Then the localisation $A_{\pp}$ is a DVR.  We denote by $k(\pp)$ the residue field of $A_{\pp}$, and for $f\in A$ by $f(\pp)$ the image of $f$ in $k(\pp)$. Now $\mathrm{Frac}(B)/\mathrm{Frac}(A)$ is a finite extension. This field extension comes with a trace map
\[
\mathrm{Tr}_{B/A}\colon \mathrm{Frac}(B)\to \mathrm{Frac}(A),
\]
 which by integrality of $B/A$ satisfies $\mathrm{Tr}_{B/A}(B)\subset A$. Let $\qq_1,\ldots,\qq_r$ be the minimal prime ideals of $B$ above $\pp$, and denote by $e_i\geq 1$ the valuation of $\pp$ in the DVR $B_{\qq_i}$.  We similarly obtain local trace maps $\mathrm{Tr}_{\widehat{B}_{\qq_i}/\widehat{A}_{\pp}}$ and $\mathrm{Tr}_{k(\qq_i)/k(\pp)}$. Denote by $\mathrm{loc}_{\rr}\colon \mathrm{Frac}(R)\to \mathrm{Frac}(\widehat{R}_{\rr})$ the localisation map at the minimal prime ideal $\rr$ of a Krull domain $R$. 
\begin{lemma}
\label{tracelemma}
   We have
 \[
 \mathrm{loc}_{\pp}\circ \mathrm{Tr}_{B/A} = \sum_i  \mathrm{Tr}_{\widehat{B}_{\qq_i}/\widehat{A}_{\pp}}\circ \mathrm{loc}_{\qq_i}.
 \]
 Moreover, for each $f\in B$ we have
 \[
 \mathrm{Tr}_{B/A}(f)(\pp)=\sum_i e_i \mathrm{Tr}_{k(\qq_i)/k(\pp)}(f(\qq_i)).
 \]
 \end{lemma}
 \begin{proof}
Upon localising $A$ and $B$ at the multiplicative subset $A\setminus\pp$, we may assume that $A$ is a DVR with maximal ideal $\pp$, and $B$ is a Dedekind domain with maximal ideals $\qq_1,\ldots,\qq_r$.

Now consider the completion of $B$ with respect to $\pp B$ and apply the Chinese Remainder Theorem:
 \begin{align}
 \label{CRT}
 \widehat{B}\colonequals \lim_{\substack{\longleftarrow\\ n}}B/\pp^n B=\prod_i \widehat{B}_{\qq_i}.
 \end{align}
  Consider
  $b\in \widehat{B}\otimes_{\widehat{A}}\mathrm{Frac}(\widehat{A})$.  We define the trace $\mathrm{Tr}_{\widehat{B}/\widehat{A}}(b)$ by taking the trace of the $\mathrm{Frac}(\widehat{A})$-linear map given by multiplication-by-$b$ on $\widehat{B}\otimes_{\widehat{A}}\mathrm{Frac}(\widehat{A})$. Here $B$ is finite and free as an $A$-module, and a basis for $B/A$ also determines a basis for $\widehat{B}/\widehat{A}$. 
 In particular, for $b\in \mathrm{Frac}(B)$, we have $\mathrm{Tr}_{B/A}(b)=\mathrm{Tr}_{\widehat{B}/\widehat{A}}(b)$, which is the trace of a block matrix on $\prod_i\mathrm{Frac}(\widehat{B}_{\qq_i})$.
 This gives the first equality.
 
 For the second equation, we begin by reducing the first modulo $\pp$.
 It then suffices to show  the equality $e_i\mathrm{Tr}_{k(\qq_i)/k(\pp)}(f(\qq_i))=\mathrm{Tr}_{\widehat{B}_{\qq_i}/\widehat{A}_{\pp}}(f)(\pp)$
 for each $i$.
To compute $\mathrm{Tr}_{\widehat{B}_{\qq_i}/\widehat{A}_{\pp}}(f)(\pp)$, we may first determine the image of $f$ in the $k(\pp)$-vector space $\widehat{B}_{\qq_i}/\pp$, and then take the trace down to $k(\pp)$.

We note that \[\widehat{B}_{\qq_i}/\pp=\widehat{B}_{\qq_i}/\qq_i^{e_i} \simeq  \widehat{B}_{\qq_i}/\qq_i\times \qq_i/\qq_i^2\times\cdots \times \qq_i^{e_i-1}/\qq_i^{e_i}\] as $k(\pp)$-vector spaces.
The right-hand side is a $k(\qq_i)$-vector space, and to compute the trace to $k(\pp)$ we may first compute the trace to $k(\qq_i)$. 
Finally, we may consider a uniformiser $t\in \widehat{B}_{\qq_i}$. If $f=\sum_{n=0}^{\infty}a_nt^n\in \widehat{B}_{\qq_i}$ then the multiplication-by-$f$ map  on $\widehat{B}_{\qq_i}/\pp$ has, by the above,  $e_i$ diagonal entries equal to $a_0$, so the trace becomes $e_ia_0=e_if(\qq_i)$. 
 \end{proof}
 When $M$ is an $R$-module and $\rr$ is a prime ideal in $R$, denote by $M_{\rr}$ the localisation of $M$ at $\rr$, and by $\widehat{M}_{\rr}$ its completion. Now suppose that $A$ and $B$ are Dedekind domains and that $K$ is a field such that the $A$ and $B$ are $K$-algebras and $A\to B$ is an embedding of $K$-algebras. 
 
 We obtain an embedding $\Omega_{\mathrm{Frac}(A)/K}\to \Omega_{\mathrm{Frac}(B)/K}$ of the spaces of K\"ahler differentials. Now given $s\in A$ (but $s\notin K)$, $\mathrm{d}s$ is a $\mathrm{Frac}(B)$-basis for $\Omega_{\mathrm{Frac}(B)/K}$, and we obtain a trace map
 \[
 \mathrm{Tr}_{B/A}\colon\; \Omega_{\mathrm{Frac}(B)/K}\to \Omega_{\mathrm{Frac}(A)/K},\;\; f\mathrm{d}s\mapsto \mathrm{Tr}_{B/A}(f)\mathrm{d}s
 \]
 which is independent of the choice of $s$. We also denote by $\mathrm{loc}_{\rr}$ the localisation map on K\"ahler differentials with respect to a prime ideal $\rr$.
 \begin{corollary}
 \label{tracecor}
The equation 
\[
 \mathrm{loc}_{\pp}\circ \mathrm{Tr}_{B/A} = \sum_i  \mathrm{Tr}_{\widehat{B}_{\qq_i}/\widehat{A}_{\pp}}\circ \mathrm{loc}_{\qq_i}
\]
also holds true on $\Omega_{\mathrm{Frac}(B)/K}$.
 \end{corollary}
 We note that traces of $n$-th roots are easily computed.
 \begin{lemma}
 \label{zerotracelemma}
Suppose that $F(\al)/F$ is a field extension of degree $n$ defined by the minimal polynomial $\al^n=a$ for some $a\in F$. Then
\[
\mathrm{Tr}_{F(\al)/F}(\al^i)=0 \text{ unless } n\mid i.
\]
\end{lemma}
\begin{proof}
This follows directly from the shape of the minimal polynomial.
\end{proof}

We now consider a map $\rho\colon X\to C$ of curves over a field $K$ and a $K$-rational point $R\in C(K)$. As divisors, we write $\rho^*R = \sum_{Q\mapsto R}e_{Q/R}Q$ as a sum of points $Q$ on $X$. Assume that we have chosen $K$ such that $Q\in X(K)$ for each $Q$ mapping to $R$. Denote by $v_Q$ the valuation of the discrete valuation ring $\widehat{\O}_{X,Q}$.

From now on, suppose that $K$ is a finite extension of $\Q_p$ for a prime $p$ of good reduction for $X$ and $C$, with ring of integers $\O_K$ and prime ideal $\pp$. Then $\rho$ extends to an $\O_K$-morphism  $\mathcal{X}\to \mathcal{C}$ between the minimal proper regular models of $X$ and $C$.

Suppose no two distinct points mapping to $R$ have equal reduction mod $\pp$ and  $p\nmid e_{Q/R}$ for all points $Q\mapsto R$. 

\begin{lemma}
\label{uniformiserlemma}
Let $t\in \widehat{\O}_{C,R}$ be a well-behaved local coordinate, and consider $Q\in X$ mapping to $R$.
Then, after base changing to an unramified extension of $K$ of degree at most $e_{Q/R}$, there exists a well-behaved local coordinate $s_Q\in \widehat{\O}_{X,Q}$ such that $s_Q^{e_{Q/R}}=t$.
Moreover, $\rho^*\colon\widehat{\O}_{C,R}\to \widehat{\O}_{X,Q}$ induces an embedding of fraction fields of degree $e_{Q/R}$ defined by the equation $x^{e_{Q/R}}-t=0$.
\end{lemma}
\begin{proof}
Consider $Q\in \rho^{-1}(\{R\})$. As $Q$ is a $K$-rational point, its residue field is $k(Q)=K$. Let $\pi\in \widehat{\O}_{X,Q}$ and $t\in \widehat{\O}_{C,R}$ be two well-behaved uniformisers, and write $e=e_{Q/R}$. We can write $\rho^*(t)=b_0\pi^e+b_1\pi^{e+1}+\cdots$, with each $b_i\in \O_K$. Define $u=b_0+b_1\pi+b_2\pi^2+\cdots\in \widehat{\O}_{X,Q}^{\times}$, so that $\rho^*(t)=u\pi^{e}$. Then the mod $\pp$ reduction of $u$ is $\widetilde{u}=\sum_{n=0}^{\infty} \widetilde{b}_n\widetilde{\pi}^n$ since $\pi$ is well-behaved. Since also $t$ is well-behaved and $\widetilde{\rho}$ still has ramification degree $e$ at $\widetilde{Q}/\widetilde{R}$ by assumption, we must have $\widetilde{b}_0\neq 0$ (c.f Section \ref{prelimsection}).  Let $k'/k$ be an extension of the residue field of $K$ containing a root of $X^e-\widetilde{b}_0$. Then $k'$ corresponds to an unramified extension $K'/K$ and by Hensel's Lemma (applicable because $b_0\in \O_K^{\times}$ and $p\nmid e$), $K'$ contains a root of $X^e-b_0$.  Then also $\widehat{\O}_{X_{K'},Q}$ contains an $e$th root of $u$ by Hensel's lemma, and we simply define $s_Q=u^{1/e}\pi$. The map between fraction fields now has degree $e$ because $\rho^*$ is an isomorphism on residue fields. 
\end{proof}
Using the lemma, we extend $K$ and define $t$ and $s_Q$ (for each $Q\mapsto R$) to be well-behaved uniformisers satisfying $s_Q^{e_{Q/R}}=t$. We consider the Dedekind domain 
\[
\O_{X,\rho^*R}\colonequals \{f\in K(X) \mid v_Q(f) \geq 0 \text{ for all } Q\mapsto R\},
\]
which is an integral extension of the DVR $\O_{C,R}$. This Dedekind domain has localisations $\O_{X,Q}$ at all places $Q$ mapping to $R$.  
Denote by $\mathrm{Tr}_{Q/R}$ the trace map from $\mathrm{Frac}(\widehat{\O}_{X,Q})$ to $\mathrm{Frac}(\widehat{\O}_{C,R})$. Now suppose that $\om$ is a global meromorphic differential on $X/K$. We can interpret $\om$ as a K\"ahler differential in $\Omega_{K(X)/K}$. From the integral extension $\rho^*\colon \O_{C,R}\to \O_{X,\rho^*R}$ of Dedekind domains, we thus obtain a trace map \[
\mathrm{Tr}\colon \; \Omega_{K(X)/K}\to \Omega_{K(C)/K},
\]
which for  $\om\in H^0(X,\Omega)$ equals the trace map defined in Section \ref{prelimsection}.
\begin{proposition}\label{coefficients conditions}
 Write  $\mathrm{loc}_Q(\om)=\sum_{i=0}^{\infty}a_i(Q)s_Q^i\dd s_Q$ for the expansion of $\om$ in $\Omega_{\mathrm{Frac}(\widehat{\O}_{X,Q})/K}$. Then $\mathrm{loc}_R(\mathrm{Tr}(\om))=\sum_{Q\mapsto R}\sum_{j\geq 1}a_{je_{Q/R}-1}(Q)t^{j-1}\dd t$. 
 
 In particular, if $\mathrm{Tr}(\om)=0$, then for each $j\geq 1$ we have
\[
\sum_{Q\mapsto R}a_{j\cdot e_{Q/R}-1}(Q)=0.
\]
\end{proposition}
\begin{remark}
Note that this equality determines a linear relation between the columns of the matrix $\widetilde{\mathcal{A}}$ defined in the statement of Theorem \ref{bigtheorem}. For example, when $\rho$ has degree 2 and $Q\in X$ ramifies, the equation says $a_1(Q)=0$, and $\widetilde{\mathcal{A}}$ has a vanishing column, c.f.\ Example \ref{firstex}.
\end{remark}
\begin{proof}
We note that $K(X)=\mathrm{Frac}(\O_{X,\rho^*R})$ and $K(C)=\mathrm{Frac}(\O_{C,R})$, and recall that the extension $\O_{X,\rho^*R}/\O_{C,R}$ is integral. We also obtain local trace maps $\mathrm{Tr}_{Q/R}\colon\;\Omega_{\mathrm{Frac}(\widehat{\O}_{X,Q})/K}\to \Omega_{\mathrm{Frac}(\widehat{\O}_{C,R})/K}$. Applying Corollary \ref{tracecor} to $\O_{X,\rho^*R}/\O_{C,R}$, we obtain
 \[
\mathrm{loc}_{R}(\mathrm{Tr}(\om))=\sum_{Q\mapsto R} \mathrm{Tr}_{Q/R}(\mathrm{loc}_Q(\om))=\sum_{Q\mapsto R}\mathrm{Tr}_{Q/R}\left(\sum_{i=1}^{\infty} a_{i-1}(Q)s_Q^{i-1}\dd s_Q\right).
\]
Next, we recall that by Lemma \ref{zerotracelemma} we have $\mathrm{Tr}_{Q/R}(s_Q^i)=0$, unless $i$ is a multiple of $e_{Q/R}$, in which case $\mathrm{Tr}_{Q/R}(s_Q^{je_{Q/R}})=e_{Q/R}t^j$. Now
$
i\mathrm{Tr}_{Q/R}(s_Q^{i-1}\dd s_Q)=\dd \mathrm{Tr}_{Q/R}(s_Q^i) ,
$
from which we deduce that $\mathrm{loc}_R(\mathrm{Tr}(\om))=\sum_{Q\mapsto R}\sum_{j=1}^{\infty}a_{je_{Q/R}-1}(Q)t^{j-1}\dd t$. 
Finally, we recall from Section \ref{prelimsection} that $\Omega_{\widehat{\O}_{C,R}/K}= K\llbracket t\rrbracket\dd t$, so the equality of power series yields a coefficient-wise equality. 
\end{proof}
 Finally, we show that traces of uniformisers can be computed as expected. 
\begin{lemma}
\label{localtracelemma}
Consider given $Q\in X(K)$ with $\rho(Q)=R$, again such that no other point mapping to $R$ has the same reduction as $Q$. Suppose $P\in D(R)$ and write, as divisors, $ \rho^*P\cap D(Q) = \sum_{i=1}^e P_i$ (with possible repetition), where $e$ is the common ramification index of $Q/R$ and $\widetilde{Q}/\widetilde{R}$. We base change $X$ and $C$ from $K$ to $K(P_1,\ldots,P_e)$, the extension of $K$ containing all coordinates of $P_1,\ldots,P_e$. Then for a well-behaved uniformiser $s$ at $Q$ we have
\[
\mathrm{Tr}_{Q/R}(s)(P)=\sum_{i=1}^e s(P_i).
\]
Consequently, if $f(X)=1-b_1X+b_2X^2-\cdots+(-1)^eb_eX^e$ is the reverse minimal polynomial of $s\in \widehat{\O}_{X,Q}$ over $\widehat{\O}_{C,R}$, then $b_i(P)$ is equal to the $i$th symmetric polynomial in $s(P_1),\ldots,s(P_e)$. 
\end{lemma}
\begin{proof}
We denote by $\widehat{\O}_{\mathcal{X},\widetilde{Q}}$ the completed local ring (and Krull domain) at the closed point $\widetilde{Q}$ and by $\widehat{\O}_{\mathcal{X},Q}$ the completed local ring (and DVR) at the non-closed point (or subscheme of codimension 1) $Q$. If $\mm_Q$ is the (minimal) prime ideal of $\widehat{\O}_{\mathcal{X},\widetilde{Q}}$ corresponding to $Q$, then $\widehat{\O}_{\mathcal{X},Q}$ is the localisation of $\widehat{\O}_{\mathcal{X},\widetilde{Q}}$ at $\mm_Q$. Well-behaved means that $s\in \widehat{\O}_{\mathcal{X},\widetilde{Q}}$. We note that the map $\rho^*\colon \widehat{\O}_{\mathcal{C},\widetilde{R}}\to \widehat{\O}_{\mathcal{X},\widetilde{Q}}$ determines an integral extension of Krull domains. We define $A=\widehat{\O}_{\mathcal{C},\widetilde{R}}$ and $B=\widehat{\O}_{\mathcal{X},\widetilde{Q}}$ and consider the minimal prime ideal $\mm_R\subset A$. To compute $\mathrm{Tr}_{B/A}(s)$, we apply the first part of Lemma \ref{tracelemma} to $A\to B$ with $\pp=\mm_R$. By assumption, $\mm_Q$ is the only minimal prime of $B$ above $\mm_R$, so we find $\mathrm{Tr}_{B/A}(s)=\mathrm{Tr}_{Q/R}(s)$. Similarly, the minimal primes above $\mm_P$ are $\mm_{P_1},\ldots,\mm_{P_e}$ (with possible repetition) and we apply the second part of Lemma \ref{tracelemma} to $A\to B$ and the ideal $\pp=\mm_P$. This yields $\mathrm{Tr}_{B/A}(s)(P)=\sum_{i=1}^es(P_i)$, as desired.  
\end{proof}

\subsection{Lemmas to control the prime bound}
The following lemmas are needed to control the lower bound on the prime $p$ in Theorems \ref{bigtheorem} and \ref{biggertheorem}. 

\begin{lemma}\label{lemma:first-inequality}
Let $N$, $T$ and $\ell>T$ be positive integers, and $p>N+T$ a prime number. Then  we have
\begin{equation}\label{InequalityOrders}
 \ell-T\geq N\ord_p(\ell)+1.    
\end{equation}
\end{lemma}

\begin{proof}
Write $\ell=p^a b$, where $a=\mathrm{ord}_p(\ell)$. When $a=0$, this is true because $\ell>T$. Otherwise, we note that $\ell\geq p^a$, and
\[
p^a-T\geq (N+T+1)^a-T\geq 1+a(N+T)-T\geq 1+aN,
\]
as desired.
\end{proof}

Next, we consider the inclusions of  polynomial rings $\Q[s_1,\ldots,s_n]=\Q[e_1,\ldots,e_n]\subset \Q[z_1,\ldots,z_n]$, where $s_k=\sum_{\ell=1}^nz_{\ell}^k$ for each $k\geq 1$ (we also make this definition for $k>n$), and $e_k$ is the $k$th elementary symmetric polynomial in $z_1,\ldots,z_n$. We adopt the conventions $e_k=0$ for $k>n$ and $e_0=1$. Let $\mm\subset \Q[e_1,\ldots,e_n]$ be the ideal generated by $e_1,\ldots,e_n$. Then in the ring of formal power series $\Q[z_1,\ldots,z_n]\llbracket X\rrbracket$, Newton's identities can be given the following compact form:
\begin{align}
    \label{newton}
\sum_{k=0}^{\infty} (-1)^ke_kX^k=\prod_{i=1}^n(1-z_iX)= \mathrm{exp}\left(-\sum_{k=1}^{\infty}\frac{s_k}{k}X^k\right).
\end{align}

\begin{lemma}
\label{prelimlemma}
\begin{itemize}
    \item[(1)] For each $k\geq 1$, the difference 
    $s_k/k -(-1)^{k-1}e_k$ is in $\mm^2$, and equals a sum of terms $b\cdot e_{i_1}\cdots e_{i_{\ell}}$ (for $\ell\geq 2$ and $i_1,\ldots,i_{\ell}\geq 1$) with coefficient $b\in \Q$ of denominator dividing $\ell$. 
   \item[(2)] If $k>n$ we have
\[
\frac{s_k}{k} -\left( \frac{(-1)^k}{2}\sum_{i=k-n}^n e_ie_{k-i} \right) \in \mm^3,
\]
and it equals again a sum of terms $b\cdot e_{i_1}\cdots e_{i_{\ell}}$ (for $\ell\geq 3$) with coefficient $b$ of denominator dividing $\ell$. In particular, if $k>2n$ then $s_k/k \in \mm^3$. 
    \item[(3)] Suppose that $S$ is a ring with an ideal $I$, and we have $a_1,\ldots,a_{2n}\in I$ satisfying for each $i\leq n$ that $a_i\equiv a_{n+i} \mod I^w$, where $w\in \Z_{\geq 1}$. Then $a_1\cdots a_n\equiv a_{n+1}\cdots a_{2n} \mod I^{w+n-1}$. 
\end{itemize}
\end{lemma}
\begin{proof}
 For part (3), it suffices to note the equality
\[
a_1\cdots a_n-a_{n+1}\cdots a_{2n}=\sum_{i=1}^n(a_i-a_{i+n})a_{i+1}a_{i+2}\cdots a_{i+n-1}.
\]
Parts (1) and (2) follow directly from (\ref{newton}) after applying $\mathrm{Log}$ to both sides of the equation and comparing coefficients for $X^k$.
\end{proof}

\subsection{Proof of Theorem \ref{bigtheorem}}
\label{proofsec}
For simplicity of exposition (to avoid using triple indices), we assume that $h=1$. Then there is one map $\rho\colon X\to C$, and we assume moreover that 
\[
\mathcal{Q}_1=\rho^*(m\cdot R) \text{ for }R\in C(K)\text{ and }m\in \Z_{>0}.
\]For the general case, one can simply sum all arguments that follow over the points in the curves $C_j$, because $\mathcal{Q}_1,\ldots,\mathcal{Q}_h$ are disjoint. We write $\rho^*R=\sum_{i=1}^k m_iQ_i$, where $m_i\geq 0$ for each $i$, and 
\[
\mathcal{Q}=\sum_{i=1}^k n_iQ_i,
\]
where $n_i\geq m\cdot m_i$ for each $i$. We choose the ordering so that $m_i=0$ for all $i>k'$, where $k'\leq k$. Note that $\mathcal{Q}_0=\sum_{i=1}^k(n_i-mm_i)Q_i$ and $\mathcal{Q}_1=\sum_{i=1}^{k'}mm_iQ_i$. 

Let $L=K(\mathcal{Q})$ be the field obtained by adjoining the coordinates of $Q_1,\ldots,Q_k$ and consider all points, maps and curves as base-changed to $L$. Let $t_R$ be a well-behaved uniformiser at $R$. Using Lemma \ref{uniformiserlemma}, at each of the distinct points $Q_i$, consider a well-behaved uniformiser $t_{Q_i}\in \widehat{\O}_{X,Q_i}$ satisfying $t_{Q_i}^{m_i}=t_R$ if $i\leq k'$. For this we need to extend $L$ by an unramified extension, which we still denote by $L$. Recall that $\mathcal{P}$ is the residue disc of $\mathcal{Q}$. Consider $F:=L(\mathcal{P})$ and let $\qq$ be  a prime in $F$ above $\pp$. We write \[\mathcal{P}=\sum_{i=1}^k(P_{i,1}+\cdots+P_{i,n_i}), \text{ where }P_{i,j}\equiv Q_i \mod \qq \text{
for each }j\in \{1,\ldots,n_i\}.
\]
We write $z_{i,j}:=t_{Q_i}(P_{i,j})$  and $s_{i,\ell}=\sum_{j=1}^{n_i}z_{i,j}^{\ell}$, and let $e_{i,\ell}$ be the $\ell$th elementary symmetric polynomial in $z_{i,1},\ldots,z_{i,n_i}$. Note that each $z_{i,j}\in \qq\O_{F_{\qq}}$ because $t_{Q_i}$ is well-behaved. Fix one annihilating differential $\om$ of trace zero for all maps and write $\mathrm{loc}_{Q_i}(\om)=\sum_{j\geq 0} a_{j}(Q_i,\om)t_{Q_i}^j\dd t_{Q_i}$. After multiplying $\om$ with a scalar, we may assume $\om\in H^0(\mathcal{X},\Omega)$ and each $a_{j}(Q_i,\om)\in \O_{L_{\pp}}$. 

For $i>k'$, we know that $e_{i,1}=\cdots=e_{i,n_i}=0$ implies that $z_{i,j}=0$ for all $j\in \{1,\ldots,n_i\}$, and hence  $P_{i,j}=Q_i$ for all $j\in \{1,\ldots,n_i\}$ by injectivity of $t_{Q_i}$ on residue discs. For $i\leq k'$, we would like to be able to read off from the $e_{i,\ell}$ whether $\mathcal{P}$ is (partially) a pullback.
\begin{lemma}\label{pull-back condition lemma}
Consider the notation as defined so far, and define the effective divisor $\mathcal{P}'=\sum_{i=1}^{k'}(P_{i,1}+\cdots+P_{i,n_i})$ (sum only up to $k'$). 

Then $\mathcal{P'}=\sum_{i=1}^{k'}(n_i-mm_i)Q_i+\rho^*\mathcal{S}$ for some $\mathcal{S}\in C^{(m)}(\overline{K})$ if (and only if) there exists a polynomial $f\in \overline{K}_{\rr}[X]$ of degree $m$ with constant coefficient 1 such that for all $i\in \{1,\ldots,k'\}$ we have an equality of polynomials
\[
1-e_{i,1}X+\cdots+(-1)^{n_i}e_{i,n_i}X^{n_i}=f(X^{m_i}).
\]
\end{lemma}
\begin{proof}
Denote by $\mathcal{R}$ the set of inverses of roots of $f$. Choose $i\in \{1,\ldots,k'\}$. By assumption, we find that
\[
\prod_{j=1}^{n_i}(1-z_{i,j}X)=f(X^{m_i}). 
\]
In particular, after reordering we find that 
\begin{itemize}
    \item[(a)] $z_{i,mm_i+1}=0,\ldots,z_{i,n_i}=0$, and moreover 
    \item[(b)] $\{z_{i,j}^{m_i}\mid j\in \{1,\ldots,mm_i\}\}=\mathcal{R}$. 
\end{itemize}
Because well-behaved uniformisers are injective on residue classes, (a) implies that $Q_i=P_{i,mm_i+1}=\cdots=P_{i,n_i}$.  Recall that $z_{i,j}^{m_i}=t_R(\rho(P_{i,j}))$ for $j\leq mm_i$. From (b) and because $t_R$ is injective on $D(R)$, we find that  $\{\rho(P_{ij}) \mid i\in \{1,\ldots,k'\}, j \in \{1,\ldots,mm_i\}\}$ has size $m$.  Denote these points by $S_1,\ldots,S_m$. They satisfy $S_{\ell}\equiv R \mod\qq$ for each $\ell$, and $\{t_R(S_{\ell}) \mid \ell \in \{1,\ldots,m\}\}=\mathcal{R}$.

Write $\rho^*(S_{\ell})=\sum_{i=1}^{k'}\sum_{j=1}^{m_i}P_{\ell ij}'$, where again $P_{\ell ij}'\equiv Q_i$  modulo a fixed prime above $\qq$ in $F(\rho^*S_{\ell})$ for each $i$. Write $z_{\ell ij}'=t_{Q_i}(P'_{\ell ij})$, and define $e_{\ell ij}'$ to be the $j$th elementary symmetric polynomial in $z_{\ell,i,1}',\ldots,z_{\ell,i,m_i}'$.  Recall that $t_{Q_i}^{m_i}=\rho^*t_R$. From Lemma \ref{localtracelemma}, we see that $e_{\ell,i,j}'=0$ for each $j  \in \{1,\ldots,m_i-1\}$ and $e'_{\ell,i,m_i}=(-1)^{m_i+1}t_R(S_\ell)$. We conclude that  
\[
\left(\prod_{\ell=1}^m\prod_{j=1}^{m_i}(1-z_{\ell,i,j}'X)\right)=\prod_{\ell=1}^m(1-t_R(S_\ell)X^{m_i})=f(X^{m_i})
\]
for each $i$, from which it follows by injectivity of $t_{Q_i}$ that \[
P_{i,1}+\cdots+P_{i,mm_i}=\sum_{\ell=1}^m(P_{\ell,i,1}'+\cdots+P_{\ell,i,m_i}'),
\]
and thus $\mathcal{P}=\sum_{i=1}^k(n_i-mm_i)Q_i+\rho^*(S_1+\cdots+S_m)$.
\end{proof}
We continue the proof of Theorem \ref{bigtheorem}. 
As $[\mathcal{P}-\mathcal{Q}]\in J(K)$ and $\om$ is a annihilating differential, we find that
$
\int_{\mathcal{Q}}^{\mathcal{P}}\om =0.
$
This is a sum of tiny integrals equal by Proposition \ref{intprops2} (i) to
\begin{equation}\label{IntegralEquationGeneral11}
\sum_{i=1}^{k}\sum_{\ell=1}^{\infty} a_{\ell-1}(Q_i,\om)\dfrac{s_{i,\ell}}{\ell}=0.
\end{equation}
For $i\in \{1,\ldots,k'\}$, write $p_{i,\ell}:=\frac{1}{m_{k'}}\sum_{j=1}^{n_{k'}}\mathrm{Tr}_{Q_i/R}(t_{Q_i}^{\ell})(\rho(P_{k',j}))$. Then $p_{i,\ell}=0$ when $m_i\nmid \ell$ and $\frac{p_{i,jm_i}}{jm_i}=\frac{s_{k',jm_{k'}}}{jm_{k'}}$ for $j\geq 1$.  Now we consider $i=k'$, and subtract
\[
0=\frac{1}{m_{k'}}\int_{n_{k'}Q_{k'}}^{P_{k',1}+\cdots+P_{k',n_{k'}}}\rho^*\mathrm{Tr}(\om)=\sum_{i=1}^{k'}\sum_{\ell=1}^{\infty}\frac{a_{\ell-1}(Q_i,\om)}{\ell}p_{i,\ell}
\]
from  (\ref{IntegralEquationGeneral11}). Here we used Proposition \ref{coefficients conditions} to evaluate the integral. 

We thus obtain
\begin{equation}
\label{eqnn}
\sum_{i=1}^{k'}\sum_{\ell=1}^{\infty} a_{\ell-1}(Q_i,\om)\left(\dfrac{s_{i,\ell}-p_{i,\ell}}{\ell}\right)+\sum_{i=k'+1}^{k}\sum_{\ell\geq 1} a_{\ell-1}(Q_i,\om)\dfrac{s_{i,\ell}}{\ell}=0.
\end{equation}

For $i\leq k'$, define $r_{i,\ell}$ for $\ell\in \Z_{\geq 1}$ by $r_{i,\ell}=0$ when $m_i\nmid \ell$ and $r_{i,jm_i}=(-1)^{jm_i-jm_{k'}}e_{k',jm_{k'}}$. For $i>k'$, define $r_{i,\ell}=0$ for all $\ell$. 
 We now use  Lemma \ref{prelimlemma} (1) to rewrite traces/power sums in terms of elementary symmetric polynomials: 
 \[
 p_{i,jm_i}/jm_i=s_{k',jm_{k'}}/jm_{k'}=(-1)^{jm_i-1}r_{i,jm_i}+\hot\] and $s_{i,\ell}/\ell=(-1)^{\ell-1}e_{i,\ell}+\hot$.
 This yields
\begin{align}\label{powerserieseqn}
0=\sum_{i=1}^{k}\sum_{\ell=1}^{n_i}(-1)^{\ell-1}a_{\ell-1}(Q_i,\om)(e_{i,\ell}-r_{i,\ell})+\hot,
\end{align}
where the higher order terms are of the form
     \begin{align}\label{hot}
    \frac{\al}{s}\left(e_{i,i_1}\cdots e_{i,i_{s}}
    -r_{i,i_1}\cdots r_{i,i_s}\right) \text{ with } s\geq 2 \text{ and } \al\in \O_{L_{\pp}}.
    \end{align}
     
 Now define $f(X)=1+(-1)^{m_{k'}}e_{k',m_{k'}}X+\cdots+(-1)^{mm_{k'}}e_{k',mm_{k'}}X^{m}$. Then for each $i\leq k'$, we find that $f(X^{m_i})=1-r_{i,1}X+\cdots+(-1)^{n_i}r_{i,n_i}X^{n_i}$. In view of Lemma \ref{pull-back condition lemma}, it is our aim to show that for each $i\in \{1,\ldots ,k'\}$, we have
 \[
 f(X^{m_i})=1-e_{i,1}X+\cdots+(-1)^{n_i}e_{i,n_i}X^{n_i}, 
 \]
 and for $i>k'$ we have $e_{i,\ell}=0$ for all $\ell\in \{1,\ldots,n_i\}$. In other words, we aim to show $e_{i,\ell}=r_{i,\ell}$ for all $i\in \{1,\ldots,k\}$ and $\ell\leq n_i$. 
 
We first show that $e_{i,\ell},r_{i,\ell}\in L_{\pp}$.  Note that $F/L$ is Galois. Suppose that $\sigma\in \mathrm{Gal}(F_{\qq}/L_{\pp})$. For each $P_{i,j}\in \mathcal{P}$ (that is, in the support of $\mathcal{P}$)  reducing mod $\qq$ to $\widetilde{P_{i,j}}=\widetilde{Q_i}$, also $P_{i,j}^{\sigma}$ must reduce to $\widetilde{Q_i}=\widetilde{Q_i^{\sigma}}$. Also $P_{i,j}^{\sigma}\in \mathcal{P}$ as $\mathcal{P}$ is $K$-rational, so $P_{i,j}^{\sigma}=P_{i,j'}$ for some $j'\in \{1,\ldots,n_i\}$. As moreover $t_{Q_i}^{\sigma}=t_{Q_i}$, we find that $e_{i,\ell}^{\sigma}=e_{i,\ell}$. By definition, then also $r_{i,\ell}^{\sigma}=r_{i,\ell}$, and all $e_{i,\ell}$ and $r_{i,\ell}$ are in $L_{\pp}$. 
We define
\[
\nu:=\mathrm{min}_{\substack{i\geq 1\\ \ell \leq n_i}}v_{\pp}(e_{i,\ell}-r_{i,\ell}),
\]
where $v_{\pp}$ is $\pp$-adic valuation. We argue by contradiction and assume that $\nu<\infty$. Recall that $\nu\geq 1$ because $\mathcal{P}$ and $\mathcal{Q}$ are in the same residue disc.

Note that in fact $e_{i,\ell}\in K(Q_i)$ and $r_{i,\ell}\in K(Q_{k'})$. Let $\pp_i$ be the prime of $K(Q_i,Q_{k'})$ below $\pp$ and let $\nu_i:=\lceil \nu/e_{\pp/\pp_i}\rceil$, where $e_{\pp/\pp_i}$ denotes the ramification index. Recall that $N\geq e_{\pp_i/p}$. We apply  Lemma \ref{prelimlemma} (3) with $I=\pp_i$, which yields $e_{i,i_1}\cdots e_{i,i_s}\equiv r_{i,i_1}\cdots r_{i,i_s}\mod \pp_i^{\nu_i+s-1}$. By Lemma \ref{lemma:first-inequality} with $T=1$, and because $p\geq e_{\pp_i/p}+2$, we find that $s-1\geq e_{\pp_i/p}\ord_{p}(s)+1$. We conclude that each of the higher order terms (\ref{hot}) vanishes mod $\pp_i^{\nu_i+1}$, hence mod $\pp^{\nu+1}$. 

Now let $\om_1,\ldots,\om_q\in H^0(\mathcal{X},\Omega)$ be linearly independent generators for the space of annihilating differentials with trace zero with respect to each $\rho_i$. We obtain equation (\ref{powerserieseqn}) for each $\om\in \{\om_1,\ldots,\om_q\}$. Let $\mathcal{A}$ be the matrix made up of the $v(\omega_j,t_{Q_i},n_i)$, so that its reduction mod $\pp$ is the matrix $\widetilde{\mathcal{A}}$ from the statement of the theorem. Denote by $\mathcal{B}$ the matrix obtained from $\mathcal{A}$ by removing for each $j\in \{1,\ldots,m\}$ the column with entries \[
((-1)^{jm_{k'}}a_{jm_{k'}-1}(Q_{k'},\om_1),\ldots,(-1)^{jm_{k'}}a_{jm_{k'}-1}(Q_{k'},\om_q)).
\]
By Proposition \ref{coefficients conditions}, the columns of $\mathcal{A}$ satisfy $m$ linear equations, each of which has coefficient 1 for exactly one of the removed columns. Since $\widetilde{\mathcal{A}}$ has rank $n-m$, we conclude that the reduction $\widetilde{\mathcal{B}}$ of $\mathcal{B}$ has full rank $n-m$. Denote by $\mathbf{v}$ the vector of length $n-m$ with as entries the list
$e_{1,1}-r_{1,1},\ldots,e_{1,n_1}-r_{1,n_1},\ldots,e_{k,n_{k}}-r_{k,n_{k}}$, from which we remove the elements $e_{k',jm_{k'}}-r_{k',jm_{k'}}$ for $j\in \{1,\ldots, m\}$ (each of which is zero by definition of $r_{i,\ell}$). Then $\mathbf{v}=0\mod \pp^{\nu}$ by definition of $\nu$, and from (\ref{powerserieseqn}) and the analysis of the higher order terms, we conclude that
\[
\mathcal{B}\cdot \mathbf{v}=0\mod \pp^{\nu+1}.
\]
Since $\widetilde{\mathcal{B}}$ has full rank, we find that $\mathbf{v}=0\mod \pp^{\nu+1}$, contradicting the definition of $\nu$, unless indeed $\nu=\infty$ and $\mathbf{v}=0$. We find that $e_{i,\ell}=r_{i,\ell}$ for all values of $i$ and $\ell$, as desired.

For $i>k'$, this means $z_{i,j}=0$ for all $j\in \{1,\ldots,n_i\}$, so that $\mathcal{P}_{i,j}=Q_i$ for each such $i$ and $j$. We then apply Lemma \ref{pull-back condition lemma} with $\mathcal{P}'=\sum_{i=1}^{k'}(P_{i,1}+\cdots+P_{i,n_i})$ and polynomial $f$. We conclude that $\mathcal{P}$ is of the desired form.\hfill $\blacksquare$

\subsection{Chabauty using the higher order terms}\label{Subsection:biggertheorem}
In some situations, there is an obstruction that prevents the matrix in Theorem \ref{bigtheorem} to have sufficiently high rank at any prime, but we do expect the outcome of the theorem to hold true. 
\begin{example}
\label{example:problem_with_73}
Let us see an example of this obstruction.
Consider the degree 2 map $\rho\colon \; X_0(73)\to X_0^+(73)$. Let $c_0,c_{\infty}\in X_0(73)(\Q)$ be the two cusps, and consider $\mathcal{Q}=3c_0$. Note that $w_{73}$ interchanges these cusps, i.e. $\rho(c_0)=\rho(c_{\infty})$. Let $t$ be the pullback under $\rho$ of a well-behaved uniformiser at $\rho(c_0)$. Then $t$ is a well-behaved uniformiser at $c_0$ and at $c_{\infty}$. The curve $X_0(73)$ has genus 5, and its quotient $X_0^+(73)$ has genus 2. We find a 3-dimensional space of annihilating differentials with trace zero on $X_0(73)$ (c.f. \cite{box}). Denote by $\om_1,\om_2,\om_3$ a basis.  By Proposition \ref{coefficients conditions}, if we write  $\mathrm{loc}_{c_0}(\om_i)=\sum_{j\geq 0} a_{i,j}t^j \dd t$ then $\mathrm{loc}_{c_{\infty}}(\om_i)=\sum_{j\geq 0}(-a_{i,j})t^j\dd t$. Consider the matrices
\[
\mathcal{A}_0=(a_{i,j})_{\substack{1\leq i \leq 3\\ 0\leq j \leq 2}} \text{ and } \mathcal{A}_{\infty}=-\mathcal{A}_0.
\]
Then $\mathcal{A}:=(\mathcal{A}_0 \mid \mathcal{A}_{\infty})$ satisfies $\mathrm{rk}(\mathcal{A})=\mathrm{rk}(\mathcal{A}_0)$. Also, $\mathcal{A}$ is the matrix whose mod $p$ reduction corresponds to the matrix in condition (3) of Theorem \ref{bigtheorem} for the point $\mathcal{Q}=3c_0+3c_{\infty}$. If its reduction $\widetilde{\mathcal{A}}$ modulo any large prime $p$ had rank 3, the entire mod $p$ residue class of $3c_0+3c_{\infty}$ would be contained in $\rho^*\left((X_0^+(73))^{(3)}(\Q)\right)$

This is not the case, however. We compute that the Riemann--Roch space $L(3c_0+3c_{\infty})$ is 3-dimensional, which is unusually large. In fact, we find a degree 6 function $f\in L(3c_0+3c_{\infty})$ such that $w_{73}^*f \neq f$, i.e. the corresponding map $f: X\to\P^1$ does not factor via $\rho$. Now note that $3c_0+3c_{\infty}=f^*(1:0)$. So for each prime $p$ of good reduction, the points $R\in \P^1(\Q)$ such that $R\equiv (1:0) \mod p$ satisfy also that $f^*R\equiv 3c_0+3c_{\infty}$. But, as $f$ does not factor via $X_0^+(73)$, $f^*R$ is in general not the pullback of a degree 3 divisor on $X_0^+(73)$. By Theorem \ref{biggertheorem}, $\widetilde{\mathcal{A}}$ therefore cannot have rank 3, and neither can $\widetilde{\mathcal{A}}_0$. Indeed, we find that $\om_3$ satisfies $a_{3,0}=a_{3,1}=a_{3,2}=0$, so that the entire bottom row of $\mathcal{A}_0$ is zero. Therefore, condition (3) in Theorem \ref{bigtheorem} is never satisfied for $\mathcal{Q}=3c_0$.

A similar situation occurs on $X_0(57)$.
\end{example} In such cases, it can help to look further into the expansion of the 1-forms. We begin with some notation. Consider a point $Q$ on a curve $X$ and a regular 1-form $\om\in H^0(X,\Omega)$. We expand $\om$ around $Q$ in terms of a uniformiser $t_Q$ at $Q$, giving $\mathrm{loc}_Q(\om)=\sum_{j\geq 0} a_jt_Q ^j$, and define
\begin{align*}
v(\om,t_Q,\ell,k)&:=\left((-1)^{\ell-1} a_\ell, (-1)^\ell a_{\ell+1},\ldots, (-1)^{k-2}a_{k-1}\right).
\end{align*}
With this notation, we have $v(\om,t_Q,k)=v(\om,t_Q,0,k)$. For integers $j$ and $i$ such that $j+1\leq i\leq 2j$ and any $x_1,\ldots,x_j$ in some field, we define
\[
\psi_i(x_1,\ldots,x_j)=\sum_{\ell=i-j}^jx_{\ell}x_{i-\ell}.
\]
\begin{theorem}\label{biggertheorem}
Consider a number field $K$, a curve $X/K$ and $\mathcal{Q}=\sum_{i=1}^k n_iQ_i\in X^{(n)}(K)$, where $n=\sum_{i=1}^k n_i$ and $Q_1,\ldots,Q_k\in X$ are distinct points. Let $\rr$ be a prime of $K$, of good reduction for $X$, containing the rational prime $p\in \rr$. Denote by $\mathcal{X}/\O_{K_{\rr}}$ a minimal proper regular model of $X/K$. Let $N$ be the ramification index of $p$ in $K(\mathcal{Q})$. 
\begin{itemize} \item[(1)] Suppose that $p\geq N+3$.
\item[(2)]
Suppose that the space $\mathcal{V}$ of annihilating differentials has dimension at least $n$, and consider linearly independent $\omega_1,\dots,\omega_n\in \mathcal{V}$ on $X$. For each $Q_i$, choose a uniformiser $t_{Q_i}$, and suppose that the $n\times n$-matrix
\[
\mathcal{A}:=\begin{pmatrix}
v(\om_1,t_{Q_1},n_1) & \cdots & v(\om_1,t_{Q_k},n_k)\\
\vdots & \ddots & \vdots \\
v(\om_n,t_{Q_1},n_1) & \cdots & v(\om_n,t_{Q_k},n_k)
\end{pmatrix}
\]
has rank $r<n$.
\end{itemize}We thus choose uniformisers $t_{\widetilde{Q}_i}$ at $\widetilde{Q}_i$, and a basis $\widetilde{\om}_1,\ldots,\widetilde{\om}_n$ for the image of $\mathcal{V}\cap H^0(\mathcal{X},\Omega)$ under reduction mod $\rr$, such that the vectors $v(\widetilde{\om}_j,t_{\widetilde{Q}_i},n_i)$ for all $1 \leq i\leq k$ and $r<j\leq n$ are zero vectors. 
\begin{itemize}\item[(3)] Assume that the $r\times n$ matrix $\widetilde{\mathcal{A}}_1$ defined below has rank $r$.
\item[(4)] Let $\F$ be the residue field of a prime in $K(\mathcal{Q})$ above $\rr$. Suppose moreover that there are no non-zero vectors  $(\widetilde{\mathbf{x}}_1,\dots,\widetilde{\mathbf{x}}_k)\in \F^n$ solving the system of equations $\widetilde{\mathcal{L}}\cdot \widetilde{\mathbf{w}}=\textbf{0}$, where 
\[
\widetilde{\mathbf{x}}_i=(\widetilde{x}_{i,1},\ldots,\widetilde{x}_{i,n_i}),\quad \widetilde{\mathcal{L}}=\begin{pmatrix} \widetilde{\mathcal{A}}_1 & 0 \\ 0 & \widetilde{\mathcal{A}}_2\end{pmatrix}, \quad
\widetilde{\mathcal{A}}_1:=\begin{pmatrix}
v(\widetilde{\om}_1,t_{\widetilde{Q}_1},n_1) & \cdots & v(\widetilde{\om}_1,t_{\widetilde{Q}_k},n_k)\\
\vdots & \ddots & \vdots \\
v(\widetilde{\om}_r,t_{\widetilde{Q}_1},n_1) & \cdots & v(\widetilde{\om}_r,t_{\widetilde{Q}_k},n_k)
\end{pmatrix},
\]
\[
\widetilde{\mathcal{A}}_2=\begin{pmatrix}-v(\widetilde{\om}_{r+1},t_{\widetilde{Q}_1},n_1,2n_1) & \cdots & -v(\widetilde{\om}_{r+1},t_{\widetilde{Q}_k},n_k,2n_k)\\
\vdots & \ddots & \vdots \\-v(\widetilde{\om}_{n},t_{\widetilde{Q}_1},n_1,2n_1) & \cdots & -v(\widetilde{\om}_{n},t_{\widetilde{Q}_k},n_k,2n_k) \end{pmatrix},
\]
and 
\[
\widetilde{\mathbf{w}}=(\widetilde{\mathbf{x}}_1,\ldots,\widetilde{\mathbf{x}}_k, \psi_{n_1+1}(\widetilde{\mathbf{x}}_1),\ldots,\psi_{2n_1}(\widetilde{\mathbf{x}}_1),\psi_{n_2+1}(\widetilde{\mathbf{x}}_2),\ldots,\psi_{2n_2}(\widetilde{\mathbf{x}}_2),\ldots, \psi_{n_k+1}(\widetilde{\mathbf{x}}_k),\ldots,\psi_{2n_k}(\widetilde{\mathbf{x}}_k))^T.
\]
\end{itemize}

Then $\mathcal{Q}\in X^{(n)}(K)$ is alone in its mod $\rr$ residue class.
\end{theorem}
\begin{remark}
We emphasize that $\mathrm{rk}(\mathcal{A})<n$ means that modulo any prime $\pp$ of $K(\mathcal{Q})$ the reduction of $\mathcal{A}$ mod $\pp$ has rank smaller than $n$, and Theorem \ref{bigtheorem} therefore cannot be applied.
\end{remark}
\begin{proof}
Let $\pp$ be a prime above $\rr$ in $L:=K(\mathcal{Q})$ with uniformiser $\pi\in \pp$. Let $\mathcal{P}=\sum_{i=1}^{k}\sum_{j=1}^{n_i} P_{i,j}\in X^{(n)}(K)$ belong to the residue disc of $\mathcal{Q}$, where the sum is arranged such that points $P_{i,j}$ reduce to the same point as $Q_i$ modulo  a prime $\qq$ of $F:=L(\mathcal{P})$  above $\pp$. For each $i$, choose $t_{Q_i}$ to be a well-behaved uniformiser at $Q_i$, and for each $i$ and $j$, define $z_{i,j}=t_{Q_i}(P_{i,j})\in \O_{F_{\qq}}$. Define $s_{i,\ell}=\sum_{j=1}^{n_i}z_{i,j}^{\ell}$, and let $e_{i,\ell}$ be the $\ell$th elementary symmetric polynomial in $z_{i,1},\ldots,z_{i,n_i}$. 

In the statement of the theorem, condition (3) implies that $\mathrm{Span}\{\widetilde{\om}_{r+1},\ldots,\widetilde{\om}_n\}$ is exactly the space of differentials $\widetilde{\om}$ on $\widetilde{X}$ satisfying $v(\widetilde{\om},t_{\widetilde{Q}_i},n_i)=0$ for all $i$.  Note that condition (4) is independent of the choice of bases for $\mathrm{Span}\{\widetilde{\om}_1,\ldots,\widetilde{\om}_r\}$ and $\mathrm{Span}\{\widetilde{\om}_{r+1},\ldots,\widetilde{\om}_n\}$. Now let $\om_1,\ldots,\om_n\in \mathcal{V}\cap H^0(\mathcal{X},\Omega)$ be linearly independent such that $v(\om_j,t_{Q_i},n_i)=0$ for all $i$ and for each $j\in \{r+1,\ldots,n\}$. Having chosen well-behaved uniformisers and integral differential forms we conclude that condition (4) holds for the reductions $\widetilde{\om}_1,\ldots,\widetilde{\om}_n$ of $\om_1,\ldots,\om_n$ modulo $\rr$. Moreover, the matrices $\widetilde{\mathcal{L}},\widetilde{\mathcal{A}}_1$ and $\widetilde{\mathcal{A}}_2$ are the mod $\pp$ reductions of corresponding matrices over $\O_{L_{\pp}}$, defined in terms of $\om_1,\ldots,\om_n$.

As in the proof of Theorem \ref{bigtheorem}, we find that each $e_{i,j}\in \O_{L_{\pp}}$ by Galois theory. Define
\[
\nu:=\mathrm{min}_{\substack{i\in \{1,\ldots,k\}\\ j\in \{1,\ldots,n_i\}}}v_{\pp}(e_{i,j})
\]
and assume that $\nu<\infty$. We aim to find a contradiction, giving $\nu=\infty$. Recall that $\nu\geq 1$ because $\mathcal{P}$ and $\mathcal{Q}$ are in the same residue disc. Now define $x_{i,\ell}$ by $e_{i,\ell}=\pi^{\nu}x_{i,\ell}$. Their mod $\pp$ reductions $\widetilde{x}_{i,\ell}$ will correspond to a non-zero solution of  $\widetilde{\mathcal{L}}\cdot \widetilde{\mathbf{w}}=\mathbf{0}$, which we shall show in two parts. 

 For  $\om\in \{\om_1,\ldots,\om_n\}$, we have at each $Q_i$ the expansion $\mathrm{loc}_{Q_i}(\om)=\sum_{\ell\geq 0}a_{\ell}(Q_i,\om)t_{Q_i}^{\ell}\dd t_{Q_i}$. As before, we obtain from $\int_{\mathcal{Q}}^{\mathcal{P}}\om =0$ the equation
\begin{equation}\label{IntegralEquation1}
0= \sum_{i=1}^k \sum_{\ell= 1}^{\infty}a_{\ell-1}(Q_i,\om)\frac{s_{i,\ell}}{\ell}=0.
\end{equation}
We again use Lemma \ref{prelimlemma} (1) to rewrite this in terms of $e_{i,\ell}$s, and apply Lemma \ref{prelimlemma} (3) and Lemma \ref{lemma:first-inequality} (with $T=1$) to the higher order terms to obtain
\begin{equation}\label{IntegralEquation3}
\sum_{i=1}^k \sum_{\ell= 1}^{n_i} (-1)^{\ell-1}a_{\ell-1}(Q_i,\om) e_{i,\ell}  \equiv 0 \pmod{\pp^{\nu+1}}.
\end{equation}
Dividing by $\pi^{\nu}$ and ranging over $\om\in \{\om_1,\ldots,\om_r\}$, we obtain $\widetilde{\mathcal{A}}_1\cdot (\widetilde{\mathbf{x}}_1,\ldots,\widetilde{\mathbf{x}}_k)^T=\mathbf{0}$, the ``upper half'' of $\widetilde{\mathcal{L}}\cdot \widetilde{\mathbf{w}}=\mathbf{0}$.

Consider now any of the remaining annihilating differential forms $\om\in \{\om_{r+1},\ldots,\om_n\}$. Recall that all $a_{\ell-1}(Q_i,\om)=0$ for $1\leq i\leq k$, $1\leq \ell\leq n_i$. Hence, using Lemma \ref{prelimlemma} (2), equation \eqref{IntegralEquation1} in this case becomes

\begin{equation}\label{IntegralEquation5}
\sum_{i=1}^k \sum_{\ell= n_i+1}^{2n_i} a_{\ell}(Q_i,\om)\cdot\dfrac{(-1)^{\ell}}{2}\sum_{m=\ell-n_i}^{n_i} e_{i,m}e_{i,\ell-m}  +\hot =0,
\end{equation}
where each of the higher order terms is of the form $\frac{b}{s}\cdot e_{i,i_1}\cdots e_{i,i_s}$ with $s\geq 3$ and $b\in \O_{L_{\pp}}$. Now $s\geq 3$ implies by Lemma \ref{lemma:first-inequality} (with $T=2$) that $v_{\pp}(e_{i,i_1}\cdots e_{i,i_s})\geq 2\nu + (s-2)\geq 2\nu+1+N\mathrm{ord}_p(s)$. We conclude that all higher order terms vanish mod $\pp^{2\nu+1}$. 
Multiplication by 2 (note that $2\notin \pp$) and division by $\pi^{2\nu}$ of (\ref{IntegralEquation5}) now yields
\[
\widetilde{\mathcal{A}}_2\cdot (\psi_{n_1+1}(\widetilde{\mathbf{x}}_1),\ldots,\psi_{2n_1}(\widetilde{\mathbf{x}}_1),\psi_{n_2+1}(\widetilde{\mathbf{x}}_2),\ldots,\psi_{2n_k}(\widetilde{\mathbf{x}}_k))^T=0,
\]
as desired. Note that $(\mathbf{x}_1,\ldots,\mathbf{x}_k)\neq \mathbf{0}\mod \pp$ by maximality of $\nu$. This solution of $\widetilde{\mathcal{L}}\cdot \widetilde{\mathbf{w}}=\mathbf{0}$ contradicts our assumption, so we must have $\nu=\infty$ and $e_{i,\ell}=0$ for all $i$ and $\ell$, so that also $z_{i,j}=0$ for all $i$ and $j$ and $\mathcal{P}=\mathcal{Q}$ by injectivity of well-behaved uniformisers on residue classes. 
\end{proof}

\section{The Mordell--Weil sieve}
\label{section_MWsieve}
\subsection{A formal description}As mentioned in Section \ref{overviewsection}, Theorem \ref{bigtheorem} can in some cases be used to compute the set of rational points on symmetric powers of curves when combined with the Mordell--Weil sieve. In this section we describe this sieve, which is similar to the sieves in \cite{siksek}, \cite{box} and \cite{box2}. 

We consider a curve $X$ with maps $\rho_i: X\to C_i$ of degree $d_i$, for $i\in \{1,\ldots,s\}$, and an integer $e$. Next, suppose we are given the following:
\begin{itemize}
    \item[(i)] A finite list of points $\mathcal{L}'\subset X^{(e)}(\Q)$. 
    \item[(ii)] A $\Q$-rational degree $e$ divisor $D_0$ on $X$.
    \item[(ii)] Explicit independent generators $D_1,\ldots,D_r$ for a subgroup $G\subset J(X)(\Q)$ and $I\in \Z_{\geq 1}$ such that $I\cdot J(X)(\Q)\subset G$. Here $r$ is the rank of $J(X)(\Q)$. 
    \item[(iii)] A list $p_1,\ldots,p_n$ of primes of good reduction for $X$.
\end{itemize}
Extend $\mathcal{L}'$ to a (possibly infinite) set $\mathcal{L}$ by adding for each point of the form $\mathcal{P}+\rho_1^*(\mathcal{R}_1)+\cdots+\rho_s^*(\mathcal{R}_s)\in \mathcal{L}'$, with $\mathcal{R}_i\in C_i^{(m_i)}(\Q)$ ($m_i\geq 0$) for each $i$, the entire set $\mathcal{P}+\rho_1^*C^{(m_1)}(\Q)+\cdots+\rho_s^*C^{(m_s)}(\Q)$ to $\mathcal{L}$. The purpose of the sieve is to show that $X^{(e)}(\Q)=\mathcal{L}$. 

To this end, we first consider the map
\[
\iota: \; X^{(e)}(\Q)\longrightarrow G,\;\; D\mapsto I\cdot [D-D_0].
\]
Next, define $A$ to be the abstract finitely generated abelian group isomorphic to $G$, with basis $e_1,\ldots,e_r\in A$, and set
\[
\phi: A\to J(X)(\Q),\;\; e_i\mapsto D_i. 
\]
For each $p\in \{p_1,\ldots,p_n\}$, we obtain a commutative diagram
\[
\begin{tikzcd}
\mathcal{L}' \arrow[r, hook] & X^{(e)}(\Q) \arrow[rr, "\iota"] \arrow[d] &  & J(X)(\Q) \arrow[d, "\mathrm{red}_p"] &  & A \arrow[ll, "\phi"'] \arrow[lld, "\phi_p"'] \\
                             & \widetilde{X}^{(e)}(\F_p) \arrow[rr, "\iota_p"]       &  & J(\widetilde{X})(\Fp)          &  &                                             
\end{tikzcd},
\]
where $\iota_p: D\mapsto I[D-\widetilde{D}_0]$, $\widetilde{D}_0$ is the reduction of $D_0$, and $\phi_p=\mathrm{red}_p\circ \phi$. 
\begin{definition}
Define $\mathcal{M}_p\subset J(\widetilde{X})(\F_p)$ to be the subset of elements of $\mathrm{red}_p(G)\cap \iota_p(\widetilde{X}^{(e)}(\F_p))$ that are either
\begin{itemize}
    \item[(1)] not in the image of $\mathcal{L}'$, or
    \item[(2)] the image of $\mathcal{Q}=\mathcal{P}+\rho_1^*(\mathcal{R}_1)+\cdots+\rho_s^*(\mathcal{R}_s)\in \mathcal{L}'$, with $\mathcal{R}_i\in C_i^{(m_i)}(\Q)$ ($m_i\geq 0$) for each $i$, such that $\mathcal{Q}$ does \emph{not} satisfy the conditions of Theorems \ref{bigtheorem} and, when applicable, \ref{biggertheorem}. 
\end{itemize}
\end{definition}

 By definition, any hypothetical point $\mathcal{Q}\in X^{(e)}(\Q)\setminus \mathcal{L}$ satisfies $\mathrm{red}_p\circ \iota(\mathcal{Q})\in \mathcal{M}_p$. We conclude the following.
\begin{proposition}[Mordell--Weil sieve]
If 
\[
\bigcap_{i=1}^n\phi_{p_i}^{-1}(\mathcal{M}_{p_i})=\emptyset
\]
then $X^{(e)}(\Q)=\mathcal{L}$. 
\end{proposition}

\begin{remark}
Naturally, this will only work in practice if the finite set $\mathcal{L}'$ contains all isolated points, as well as sufficiently many points (partially) composed of pullbacks. 
\end{remark}

\subsection{Implementing the sieve efficiently}
The groups $J(\widetilde{X})(\F_p)$ can be computed using the class group algorithm of Hess \cite{hess}. When $e$ and $p$ are large, it is important to implement the sieve in an efficient way. Rather than first computing $\mathcal{M}_{p_i}$ for each $i$ and intersecting afterwards, we compute intersections recursively. It also turns out that computing $\widetilde{X}^{(e)}(\F_p)$ and $\iota_p$ for large values of $e$ and $p$ is suboptimal, so we compute Riemann--Roch spaces instead.

For a divisor $D$ on a curve $Y$ over a field $K$, define the Riemann--Roch space 
\[
L(D):=\{f\in K(Y)^{\times} \mid \mathrm{div}(f)+D\geq 0\}\cup \{0\}.
\]
Also define $H_{i-1}:=\mathrm{ker}(\phi_{p_1})\cap\cdots\cap \mathrm{ker}(\phi_{p_{i-1}})$. Suppose that we have computed the finite set $W_{i-1}$ of $H_{i-1}$-coset representatives for $\cap_{j=1}^{i-1}\phi_{p_j}^{-1}(\mathcal{M}_{p_j})$, and we want to compute $W_i$. First, we determine the (larger) set $W_i'$ of $H_i$-coset representatives for the same intersection. For each $w\in W_i'$, we then compute the $\F_p$-vector spaces
\[
L(z+\widetilde{D}_0) \text{ for each } z \text{ such that } I\cdot z=\phi_{p_i}(w). 
\]
Often these spaces will simply be 0-dimensional and $w$ can be removed from $W_i'$. For each non-zero $f\in L(z+D_0)$, we obtain $\mathrm{div}(f)+z+\widetilde{D}_0\in \widetilde{X}^{(e)}(\F_p)$, and we verify whether it is in the image of $\mathcal{L}'$ and satisfies the conditions of Theorem \ref{bigtheorem} or Theorem \ref{biggertheorem}. This way, we determine $W_i\subset W_i'$ without the need to compute $\widetilde{X}^{(e)}(\F_p)$. \\

We mention another improvement. When $p_1$ and $p_2$ are distinct primes such that $\#J(X)(\F_{p_1})$ and $\#J(X)(\F_{p_2})$ are coprime, we have $\phi_{p_2}(W_1)=\phi_{p_2}(A)$ by the Chinese Remainder theorem. One thus tends to choose primes $p$ such that the numbers $\#J(\widetilde{X})(\F_p)$ have prime factors in common. Often these factors are powers of small primes, whereas $\#J(\widetilde{X})(\F_p)$ tends to also be divisible by a large prime factor, let us call it $r$. It is unlikely that $r$ will also divide $\#J(\widetilde{X})(\F_q)$ for any other prime $q$ we consider, hence the ``mod $r$" information is of little use. It does, however, slow the sieve down considerably, because information is stored in  $H_i$-cosets and the factor $r$ increases the index of $H_i$. We  remedy this by composing $\iota_p$ and $\phi_p$ with the multiplication-by-$r$ map $m_r$, thereby ``forgetting'' the mod $r$ information.

\section{Applying Chabauty to modular curves} 
For each $X_0(N)$ with $N\in \{53,61,65,67,73\}$, let $\rho\colon X_0(N)\to X_0^+(N)$ be the quotient map by $w_N$. 
Denote by $X_0^*(N)$ the quotient of $X_0(N)$ by the full Atkin-Lehner group.
To determine the cubic points on these curves, we use the Mordell--Weil sieve described in Section \ref{section_MWsieve} in combination with Theorems \ref{bigtheorem} and \ref{biggertheorem} to determine the isolated points in $X_0(N)^{(3)}(\Q)$, i.e. those not of the form $P+\rho^*Q$ with $P\in X_0(N)(\Q)$ and $Q\in X_0^+(N)(\Q)$.
For $X_0(57)$, we determine the entire set $X_0(N)^{(3)}(\Q)$.
On $X_0(65)$, we moreover determine the points in $X_0(65)^{(4)}(\Q)$ that are not of the form $\rho^*D$ for $D\in X_0^+(65)^{(2)}(\Q)$ or of the form $P+\rho^*Q$ with $P\in X_0(65)^{(2)}(\Q)$ and $Q\in X_0^+(65)(\Q)$. 

In this section, we determine the information necessary to apply Theorems \ref{bigtheorem} and \ref{biggertheorem} and run the sieve. 

\subsection{Models, points and differentials}
All the modular curves we consider were studied in \cite{box}, so we use the models, subgroups of Mordell--Weil groups and annihilating differentials as computed there. 

We recall that these (canonical) models for each $X_0(N)$ were computed in \cite{box} using either the Small Modular Curves package in \texttt{Magma} or the code written by \"Ozman and Siksek \cite{ozman}. Both algorithms simply compute equations satisfied by $q$-expansions of weight 2 cusp forms, following Galbraith \cite{galbraith}. 

To find cubic and quartic points on a modular curve $X_0(N)\subset \P^n$ (with coordinates $x_0,\ldots,x_n$), we have used code written by \"Ozman and Siksek \cite{ozman} to decompose intersections $H\cap X_0(N)$ of hyperplanes $H\subset \P^n$ given by 
\[
H:\; a_0x_0+\cdots +a_nx_n=0,
\]
with each $a_i\in \Z$ satisfying $|a_i|\leq C$ for a (small) bound $C$. Each irreducible component of $X_0(N)\cap H$ of degree $d$ yields a point of degree $d$ on $X_0(N)$.  

Decomposing such intersections can be time-consuming on genus 5 curves, and we did not always find all points this way. In such cases, we ran the sieve described in Section \ref{section_MWsieve} with this too small set of known points $\mathcal{L}'$. Consider the notation introduced in Section \ref{section_MWsieve}. After a number of primes $p_1,\ldots,p_n$, we found a non-empty set $W_n$ of representatives for the possible $H_n$-cosets where unknown rational points can map into. As the index of $H_n$ in $A$ is large, it is a priori unlikely that one of these cosets $w+H_n$, for $w\in W_n$, contains a ``small vector'', say with all coefficients for free generators of $A$ smaller than 100 in absolute value. So, heuristically, small vectors should only occur with good reason, for example as images of rational points. 

Using the LLL-algorithm \cite{LLL}, we searched for small vectors in these cosets $w+H_n$. For each small vector $v\in w+H_n$ found, we compute $L(D_0+\phi(v))$ and find that it is indeed 1-dimensional, generated by, say, $f$. Then in each case, indeed $\mathrm{div}(f)+D_0+\phi(v)$ was a new isolated rational point on the symmetric power.

Finally, in order to apply Theorems \ref{bigtheorem} and \ref{biggertheorem} to $X_0(N)$, we need for each prime $p$ (of good reduction for $X_0(N)$) a basis for the image $\widetilde{\mathcal{V}}_0$ of $\mathcal{V}_0\cap H^0(\mathcal{X}_N,\Omega)$ under the reduction map, where $\mathcal{X}_N/\Z_p$ is a proper minimal model of $X_0(N)$ and $\mathcal{V}_0$ is the space of annihilating differentials, with trace zero to $X_0^+(N)$ when $N\neq 57$ and $X_0^*(57)$ otherwise. In \cite[Section 3.4]{box}, it was shown for each $N\in \{53,61,65,67,73\}$ that $\mathcal{V}_0=\mathrm{Ker}(1+w_N^*)$ and $\widetilde{\mathcal{V}}_0=\mathrm{Ker}(1+\widetilde{w}_N^*)$, where $w_N^*$ and $\widetilde{w}_N^*$ are the pull-back morphisms on $\Omega_{X_0(N)/\Q}$ resp. $\Omega_{\widetilde{X}_0(N)/\F_p}$. Similarly, for $N=57$, it was shown that $\mathcal{V}_0=\mathrm{Ker}(1+w_{19}^*)\cap \mathrm{Ker}(1+w_{57}^*)$ and $\widetilde{\mathcal{V}}_0=\mathrm{Ker}(1+\widetilde{w}_{19}^*)\cap \mathrm{Ker}(1+\widetilde{w}_{57}^*)$. This allows us to compute those annihilating differentials and verify the conditions of Theorems \ref{bigtheorem} and \ref{biggertheorem} in practice. 

\subsection{Computing Mordell-Weil groups}
\label{subsection_MWgroups}
For each $N \in\{53,57,61,65,67,73\}$, the first named author \cite{box} computes a finite index subgroup of $J_0(N)(\Q)$.
Apart from $N=57$, the index of these subgroups is shown to divide $2$.
Thanks to our Mordell-Weil sieve \S \ref{section_MWsieve}, this almost completely suffices for our purposes.
The only cases where they do not are $N=57,65$.
For $N=57$ we produce a subgroup with index dividing $2$; this appears to be needed purely for computational reasons.
Whereas when $N=65$ the issue is somewhat deeper.
In this case we verify the subgroup given in \cite{box} is the entire Mordell-Weil group.

The method used in \cite{box} to compute subgroups of $J_0(N)(\Q)$ with bounded index can be broken down into two steps.
First, one finds a subgroup of bounded index in the free part of $J_0(N)(\Q)$.
Then one computes the torsion subgroup of $J_0(N)(\Q)$.
This is done by combining theorems of Manin--Drinfeld and Mazur with bounds given by $J_0(N)(\F_p)$ for $p\nmid N$, see \cite[Lemma 3.2]{box}.
We mention that bounds may also be obtained using Hecke operators as in \cite[Section 4.3]{dzb}, although this is not needed here.

Let $X/\Q$ be a (projective, non-singular) curve and $\Gamma \leq \Aut_\Q(X)$ a finite subgroup.
The quotient curve $C=X/\Gamma$ is also defined over $\Q$ and the natural map $\rho : X \rightarrow C$ has degree $\#\Gamma$.
Denote the Jacobian of $X$ by  $J(X)$ and that of $C$ by $J(C)$.
Choosing compatible base points for the maps $\iota_X: X \rightarrow J(X) $, $\iota_C: C \rightarrow J(C) $, we obtain a commutative diagram:
\sqcommdiag{X}{J(X)}{C}{J(C).}{\rho}{\iota_X}{\rho_*}{\iota_C}
If the ranks of $J(C)(\Q)$ and $J(X)(\Q)$ are equal then $\rho^*J(C)(\Q)$ gives a subgroup of index dividing $\#\Gamma$ in the free part of $J(X)(\Q)$ \cite[Prop. 3.1]{box}. 

The problem of computing subgroups of bounded index in $J(X)(\Q)$ is thus reduced to computing the free part of $J(C)(\Q)$.
When $C$ is an elliptic curve this may be accomplished via Cremona's method \cite{cremona}.
Else, if $C$ has genus 2, this may be achieved using techniques due to Stoll \cite{stoll}.

We shall always take $X = X_0(N)$ and $C = X_0^+(N)$ with $N \in \{53,57,61,65,67,73\}$. 
For all these values of $N$, the curve $X_0^+(N)$ has genus one or two.

\begin{example}Let us give a few more details in the case $X= X_0(57)$.
Here, $C= X^+_0(57)$ has a model defined by $y^2 = x^6 - 2x^5 + 3x^4 + 3x^2 - 2x + 1$.
Stoll's method \cite{Stoll_Chab_sans_MW} may be used to show the free part of $J_0^+(57)(\Q)$ is generated by $Q=\infty^+- \infty^-$.
Pulling $Q$ back to $J_0(57)$ gives $\rho^*(Q) = (1 : 1 : 0 : 1 : 0) + (6 : 9 : -1 : 7 : 2) - P - \overline{P}$, where \[P = \left(-\sqrt{-2}  + 4 : -4\sqrt{-2}  + 7 : \sqrt{-2} -1  : -2\sqrt{-2} +2 \right).\]
\end{example}
\subsubsection{The Mordell-Weil group of $J_0(65)$}
Let $\rho\colon X_0(65) \rightarrow X_0^+(65)$ be the quotient map. The latter is an elliptic curve and so we may determine its Mordell-Weil group.
The group
\[ G = \rho^*(J_0^+(65)(\Q)) J_0(65)(\Q)_{\mathrm{tors}}\]
 is isomorphic to $\Z \times \Z/2\Z \times \Z/84\Z$ as shown in \cite[Section 4.5]{box}.
Denote by $D$ a generator of the free part of $G$.
We use the Abel--Jacobi map $\iota: X_0(65)(\Q)^{(4)} \rightarrow J_0(\Q)$ given by $P \mapsto [P- 2\rho^*(0:1:0)]$.

A priori $G \subset J_0(65)(\Q)$ has at most index two.
Whilst the sieve described in \S \ref{section_MWsieve} normally allows one to get away with a subgroup of known finite  index, this appears not to be the case here.
Indeed, for every prime $p$ there seems to be a point on $X_0(65)(\Fp)^{(4)}$ which under the Abel--Jacobi map $\iota_p$ (compatible with $\iota$ and reduction modulo $p$) doubles to the reduction of $\rho^*([(-1:0:1)-(0:1:0)])$ (there also appear to be similar problems for four points coming from the quotient $X_0^+(65)$).
It seems likely that for each prime $p$ there is an element in the set of preimages of $\rho^*([(-1:0:1)-(0:1:0)])$ in $J_0(65)(\overline{\Q})$ under multiplication by two whose reduction belongs to $\iota_p(X_0(65)(\Fp)^{(4)})$. We have not been able to compute the full set of preimages to verify this, but this does suggest the following method to prove $G = J_0(65)(\Q)$.

First find a divisor $D_0$ defined over an extension $K/\Q$ such that $2D_0 = D$.
Computing the 2-torsion of $J_0(65)(K)$ then allows us to determine all points in $J_0(65)(K)$ which double to $D$.
Indeed, if $2D_0' = D$, then $2(D_0-D_0')=0$.
Given this set, it then suffices to decide whether any element belongs to $J_0(65)(\Q)$.

Owing in part to $J_0^+(65)$ being an elliptic curve, it was straightforward for us to find such a $D_0$.
There are algorithms and explicit formulas which compute Mumford representations for preimages under multiplication by $2$ on odd degree hyperelliptic jacobians \cite{Stoll_Chab_sans_MW, Zarhin_division_by_two_on_odd_hyperelliptic_jacs}.
In principle, this allows one to adapt our method to other curves. 

\begin{theorem}
The Mordell-Weil group $J_0(65)(\Q)$ of $X_0(65)$ is the group generated by the two subgroups \[\rho^*(J_0^+(65)(\Q)) \text{ and } J_0(65)(\Q)_{\mathrm{tors}}.\]
\end{theorem}

\begin{proof}
Let us abbreviate $J_0(65)$ to $J$.
Set $G= \langle\rho^*(J_0^+(65)(\Q)), J(\Q)_{\mathrm{tors}}\rangle$.
In \cite{box}, it is shown that $G \cong \Z \times \Z/2\Z \times \Z/84\Z$.
Explicit generators are given (they are the same as in \S \ref{subsubsection_65cubic}), let us denote them by $D,T_2,T_{84}$ respectively.
In particular, $D = \rho^*\left((1:0:1)-(0:1:0)\right)$ generates the free part of $G$ (here we use the models as given in \S \ref{subsubsection_65cubic}). By Proposition \cite[Prop. 3.1]{box} and the above it suffices to show no element in $D+\langle T_2,T_{84}\rangle$ is the double of a point in $J(\Q)$, or equivalently, none of $D$, $D+T_2$, $D+T_{84}$ and $D+T_2+T_{84}$ are doubles.
The last three may easily be checked by reducing modulo $7$.
We now focus our attention on $D$.

There are four points which double to $(1:0:1)$ on the elliptic curve $X^+_0(65)$. Two of these are defined over $\Q(\sqrt{5})$, the other two over  $\Q(\sqrt{13})$.
Let $d_0$ be any of these points.
The pullback $D_0 := \rho^*(d_0-(0:1:0))$ then satisfies $2D_0=D$.

We claim $J(K)_{\mathrm{tors}}=J(\Q)_{\mathrm{tors}}$ for $K=\Q(\sqrt{5}), \Q(\sqrt{13})$.
To verify this, we compute $J(\Fp)$ as an abstract group for primes $p \neq 5, 13$ split in $K$.
For these primes we have embeddings of abstract groups $J(K)_{\mathrm{tors}}\hookrightarrow J(\Fp)$ which allow us to show the lower bound given by $J(\Q)_{\mathrm{tors}} \cong \Z/2\Z \times \Z/84\Z$ is tight.
The primes 11 and 19 split in $\Q(\sqrt{5})$ and the corresponding groups give
\[
J(\F_{11})=\Z/2\Z\times \frac{\Z}{2^2\cdot 3\cdot 5\cdot 7^2\cdot 37\Z} \text{ and } J(\F_{19})=\frac{\Z}{2\cdot 3\cdot 23\Z}\times \frac{\Z}{2^2\cdot 3\cdot 7 \cdot 13\cdot 23},
\]
from which it is quickly deduced that $J(\Q(\sqrt{5}))_{\mathrm{tors}}=\Z/2\Z\times \Z/84\Z$.
For $K=\Q(\sqrt{13})$ one may use 17 and 23.

We thus have the equality of sets
\[S := 
\{x\in J(K) \mid 2x=D\}=\{D_0+T \mid T\in J(\Q)[2]\}.
\]
Hence we are left to show $D_0$ is not defined over $\Q$.
Denote by $d_0^c$ the Galois conjugate of $d_0$. Then both $d_0+d_0^c$ and $2d_0$ are in $X^+_0(65)(\Q)$, so also $Q:=d_0-d_0^c\in X^+_0(65)(\Q)$. Now $Q=Q^c=-Q$, so $Q$ is the unique non-trivial point in $X^+_0(65)(\Q)[2]$. We find that $D_0-D_0^c=\rho^*Q$, and it suffices to show that $\rho^*Q\neq 0 \in J(\Q)$. We can check this modulo 7. 

Hence none of the elements in $S$ is in $J(\Q)$, and so we must have that
$
J(\Q)=G,
$
as desired.
\end{proof}

\subsection{Two examples}\label{Subsection:Examples}
We now give explicit examples of Theorems \ref{bigtheorem} and \ref{biggertheorem} in use.

\begin{example}\label{firstex}
Consider the non-cuspidal rational point $Q$ ramified under the map $\rho\colon X_0(67) \rightarrow X_0^+(67)$.
This gives rise to the point $3Q$ in $X_0(67)(\Q)^{(3)}$.
We will apply Theorem \ref{bigtheorem} with $\mathfrak{r}=19$ to show any other element in its residue disc also belongs to the set $Q +\rho^*X_0^+(67)(\Q)$.
Writing $3Q= Q + \rho^*\rho(Q)$, we see it is enough to show the matrix $\widetilde{\mathcal{A}}$ has rank two.
Computing $\widetilde{\mathcal{A}}$ with respect to a uniformiser $s$ at $Q$ such that $s^2=\rho^*t$, where $t$ is a uniformiser at $\rho(Q)$, we find
\[ \widetilde{\mathcal{A}} =
\begin{pmatrix}
15 & 0 & 0 \\
2 & 0 & 2 \\
6 & 0 & 11
\end{pmatrix}
\]
which clearly has rank 2. Note the second column is  identically zero, in agreement with Proposition \ref{coefficients conditions}.
\end{example}

\begin{example}
As explained in Example \ref{example:problem_with_73}, Theorem \ref{bigtheorem} will always fail for the points $3c_0,3c_\infty \in X_0(73)(\Q)^{(3)}$, where $c_0,c_\infty \in X_0(73)(\Q)$ are the cusps.
To fix this we need expand further into the coefficients of the differentials. We verify the conditions of Theorem \ref{biggertheorem} at $\rr=19$.
Given a choice of uniformiser and differentials, the corresponding matrices for $3c_\infty$  are
\[\mathcal{A} = 
\begin{pmatrix}
  2 & -13 & 47\\
  0  & -1  &11\\
  0  & 0  & 0
\end{pmatrix}
\]
which clearly has rank 2, and the $\F_{19}$-matrices 
\[
\widetilde{\mathcal{A}}_1 =
\begin{pmatrix}
9 & 1 & 1 \\
0 & 16 & 9 
\end{pmatrix}
\text{ and }
\widetilde{\mathcal{A}}_2 =
\begin{pmatrix}
2 & 6 & 3
\end{pmatrix}
.
\]
One then checks $ \widetilde{\mathcal{L}}\cdot (x,y,z,y^2+2xz,2yz,z^2)^T=\mathbf{0}$ has no non-zero solutions for $x,y,z \in \F_{19}$.

Note Theorem \ref{biggertheorem} does not require one to form $\widetilde{\mathcal{A}}_1$ from the reductions of the uniformiser and differentials used for computing $\mathcal{A}$.
For this reason, $\widetilde{\mathcal{A}}_1$ above does not coincide with the reduction of the first two rows of $\mathcal{A}$.
\end{example}

\section{Results}
\label{section_results}
\subsection{$\Q$-curves}
\label{subsection_Qcurves}
An elliptic curve $E$ defined over a number field $K$ is said to be a $\Q$-curve if it is $\overline{\Q}$ isogenous to all of its $\Gal(\overline{\Q}/\Q)$ conjugates.
For example, all elliptic curves with rational $j$-invariant are $\Q$-curves, as are CM elliptic curves (for elliptic curves with CM by the maximal order this may be deduced from Propositions 1.2 and 2.1 in \S II of \cite{advancedsilverman}).

For $N \in \{53,57,61,65,67,73\}$ there are only finitely many cubic points on $X_0(N)$.
These points are listed \S \ref{subsection_cubicpoints}.
We find these points give rise to $\Q$-curves exactly when the elliptic curves have CM.

\begin{proposition}
 Let $K/\Q$ be a non-trivial extension of odd degree.
 Then non-CM points in $X_0(N)(K)$ for $N\in \{53,57,61,65,67,73\}$ do not give rise to $\Q$-curves.
\end{proposition}

\begin{proof}
 Let $P \in X_0(N)(K)$ be a non-CM point. While this is a well-known fact, we explain why we can take a representative $(E,C)$ of $P$ such that $E$ is defined over $K$ and $C\subset E[N]$ is a $\mathrm{Gal}(\overline{K}/K)$-stable cyclic subgroup of order $N$. Take any representative $(E,C)$ of $P$. The $j$-map $j\colon X_0(N)\to X(1)$ is defined over $\Q$. Since $j(E)=j(P)\in K$, we may take a different representative $(E',C')$ where $E'$ is defined over $K$ (e.g. by considering the Legendre form). As $P$ is a non-CM point, $E'$ is a quadratic twist of $E$ (see e.g. \cite[Lemma A.4]{QcurvesCremonaNajman}, \cite[\S X, Prop. 5.4]{SilvermanI}).
 The mod $\ell$ representations ($\ell$ prime) of $E'$ and $E$ are either equal or differ by a quadratic twist, so in particular $E'$ has a $\mathrm{Gal}(\overline{K}/K)$-stable subgroup of order $\ell$ if and only if $E$ has one also. In other words, we may choose $C'$ to be $\mathrm{Gal}(\overline{K}/K)$-stable.
 For $N \neq 65$, the statement now follows from  applying \cite[Theorem 1.1]{QcurvesCremonaNajman} to $E'$.
 
Suppose $P \in X_0(65)(K)$ is a non-CM point which gives rise to a $\Q$-curve.
Take the representative $E'/K$ of $P$ as above.
Then by Theorem 2.7 of \cite{QcurvesCremonaNajman}, $E'$ is isogenous to an elliptic curve with rational $j$-invariant and thus has rational isogenies of degrees $5$ and $13$, by Corollary 3.4 in \cite{QcurvesCremonaNajman}.
But $X_0(65)(\Q)$ consists of cusps, giving a contradiction.
\end{proof}

Cremona and Najman describe an algorithm to determine whether an elliptic curve is a $\Q$-curve or not in \S 5.5 of \cite{QcurvesCremonaNajman}. We use this algorithm to show the isolated quartic points on $X_0(65)$ which give rise to $\Q$-curves are exactly the CM points. 
This algorithm has been implemented in \texttt{Sage} \cite{QcurvesCremonaNajman}.
However, as part of this implementation the conductor is computed, and this appears to be too costly for our elliptic curves.
Instead, we partially implemented the algorithm in \texttt{Magma} and avoid computing the conductor.
In particular, for each of the non-CM elliptic curves $E/K$ under consideration there was a rational prime $p$ and primes $\pp,\qq$ of $K$ above $p$ such that either $E$ had bad potentially multiplicative reduction at $\pp$ but not $\qq$ or $E$ had good reduction at all primes above $p$ and the endomorphism algebras of the reductions at $\pp$ and $\qq$ were not isomorphic.
These contradict certain isogeny properties satisfied by $\Q$-curves, see Propositions 5.1 and 5.2 of \cite{QcurvesCremonaNajman}.

The curve $X_0(65)$ has infinitely many quartic points: the quotient $X_0^+(65)$ is an elliptic curve.
That is, we have a degree two map to an elliptic curve, and elliptic curves have infinitely many degree two maps to $\P^1$.
This yields infinitely many degree four maps to $\P^1$ and in particular infinitely many quartic points on $X_0(65)$.

In fact, by the Existence Theorem in Brill--Noether Theory \cite[Chapter V, pg206]{ACGH_book}, genus 5 curves have  infinitely many degree four maps to $\P^1$.
The fact $X_0(65)$ admits a degree two map to an elliptic curve is our saving grace in this instance.
Indeed, each such map turns out to factor through this elliptic curve $X_0^+(65)$, allowing us to perform our relative Chabauty methods with respect to the quotient.
On the flip side, no such quotient exists for $X_0(57)$, leaving us with an abundance of quartic points from infinitely many sources. 

Accordingly, every quartic point on $X_0(65)$ either arises from a quadratic point on $X_0^+(65) $, or is a point listed in \S \ref{subsection_quarticpoints}.
For points not coming from the quotient, we find the $\Q$-curves are exactly those with CM, just as the case was for the cubic points.
The points arising from the quotient may be $\Q$-curves (for example if they come from the full Atkin-Lehner quotient $X_0^*(65)$), but are not necessarily.
However, each of these points will give rise to a $K$-curve for some quadratic field $K$. Indeed, $w_{65}$ determines the isogeny between such a quartic point and one of its Galois conjugates.

\subsection{Cubic Points}
\label{subsection_cubicpoints}
Let $N \in \{53,57,61,65,67,73\}$. In this section we list all cubic points on $X_0(N)$, or rather a representative from each $\Gal(\overline{\Q}/\Q)$-conjugacy class.
As explained in \S \ref{subsection_MWgroups}, subgroups $G \subseteq J_0(N)(\Q)$ of index at most two are computed using $J_0^+(N)(\Q)$. We provide explicit generators for $G$ and $J_0^+(N)(\Q)$  below.

For $N\neq 57$, we apply the Mordell-Weil sieve in conjunction with Theorems \ref{bigtheorem}, \ref{biggertheorem} relative to $X_0^+(N)$ in order to provably show all points on $X_0^{(3)}(N)(\Q)$ either arise as the sum of a rational point and the pullback of a rational point from $X_0^+(N)$, or belong to a finite set.
This finite set consists of ($\Gal(\overline{\Q}/\Q)$-stable sums of) the cubic points and sums of rational and isolated degree two points.
Rational and isolated degree two points are listed in \cite[\S 4]{box}.

We note for $N=57$ there is no need to use `relative Chabauty' as there are only finitely many quadratic points on $X_0(57)$ and these can be determined provably using Chabauty.
In fact Siksek's symmetric Chabauty Theorem \cite[Theorem 3.2]{siksek} almost suffices to determine the cubic points.
But, as mentioned in Example \ref{example:problem_with_73}, there are two points on $X_0^{(3)}(57)$ which require Theorem \ref{biggertheorem}.
The only other curve where we are also required to apply Theorem \ref{biggertheorem} is $X_0(73)$.

Using the same models as in \cite[\S 4]{box}, we list a representative for every Galois conjugacy class of cubic points on $X_0(N)$, along with the corresponding $j$-invariant and indicate whether the class of elliptic curves have CM.
The only points which give rise to $\Q$-curves are the CM points, see \S \ref{subsection_Qcurves}.
For this reason we do not include an extra column to indicate which points are $\Q$-curves.

We recall that the \texttt{Magma} \cite{magma} code to verify all computations can be found at
\[
\texttt{\href{https://github.com/joshabox/cubicpoints/}{https://github.com/joshabox/cubicpoints/}} \;.
\]
\afterpage{%
\clearpage%

\pagestyle{plain} 
\begin{landscape}
\subsubsection{$X_0(53)$}
Model for $X_0(53)$: 
\begin{align*}
&9x _0^2  x_3 - 5  x_0  x_3^2 - 27  x_1^3 - 18  x_1^2  x_2 + 78  x_1^2  x_3 - 18  x_1  x_2^2 + 30 x_1  x_2  x_3 - 64  x_1  x_3^2 - 57  x_2^3 + 136  x_2^2  x_3 - 104  x_2  x_3^2 + 
    49  x_3^3=0,\\
    & x_0  x_2 - 2  x_0  x_3 - 3x_1^2 + 5  x_1  x_3 - 2 x_2^2 +   x_2  x_3 - 2  x_3^2=0.
\end{align*}
Genus $X_0(53)$: 4.\\
Cusps: $(1:0:0:0)$, $(1:1:1:1)$.\\
$X_0^+(53)$: elliptic curve $y^2 + xy + y = x^3 - x^2$ of conductor 53.\\\
Group structure of $J_0^+(53)(\Q)$: $\Z\cdot [Q-O]$, where $Q:=(0:-1:1)$.\\
Group structure of $G\subset J_0(53)(\Q)$: $G= \Z\cdot D_1\oplus \Z/13\Z\cdot D_{\mathrm{tor}}$, where $D_{\mathrm{tor}}=[(1:1:1:1)-(1:0:0:0)]$ and $D_1=[P+\overline{P}-(1:1:1:1)-(1:0:0:0)]=\rho^*[Q-O]$ for 
\[
P:=\left( 0, \frac16(-\sqrt{-11} + 5): 1: 1 \right)\in X_0(53)(\Q(\sqrt{-11}))
\]
satisfying $\rho(P)=Q$.\\
Primes used in sieve: 31, 17.

\begin{table}[h!]
    \centering
    \begin{tabularx}{0.82\paperheight}{>{\centering\arraybackslash}p{0.37\paperheight} >{\centering\arraybackslash}m{0.37\paperheight} >{\centering\arraybackslash}m{0.06\paperheight}}
    \multicolumn{3}{l}{Points with $\alpha$ satisfying $\alpha^3 - \alpha^2 + 5\alpha + 52 =0 $.} \\ [0.8ex]
    Coordinates & $j$-invariant & CM   \\ [0.5ex] 
        \hline \hline \rule{0pt}{2.8ex}

$(12\alpha^2 - 59\alpha + 254 : -9\alpha^2 + 22\alpha + 32 : 7\alpha^2 - 27\alpha
    + 252 : 178)$&
{$\substack{1/19383245667680019896796723(-383791426307028167821262180847852905\alpha^2 \\+
    1545152415129700179443672757063109789\alpha -
    6594966442722213746931150662879063633)}$}
& NO \\

$(\alpha^2 - 2\alpha + 7 : 3\alpha - 3 : 9 : 0)$&
$1/3(-15648377\alpha^2 + 63003373\alpha - 268896737)$
& NO \\
    \end{tabularx}
\end{table}

\begin{table}[h!]
    \centering
    \begin{tabularx}{0.82\paperheight}{>{\centering\arraybackslash}m{0.37\paperheight} >{\centering\arraybackslash}m{0.37\paperheight} >{\centering\arraybackslash}m{0.06\paperheight}}
    \multicolumn{3}{l}{Points with $\alpha$ satisfying $\alpha^3 + 11 \alpha - 8 =0 $.} \\ [0.8ex]
    Coordinates & $j$-invariant & CM   \\ [0.5ex] 
        \hline \hline \rule{0pt}{2.8ex}

$(143\alpha^2 + 63\alpha + 1740 : 53\alpha^2 + 49\alpha + 1004 : 25\alpha^2 +
    33\alpha + 612 : 524)$&
{$\substack{1/9007199254740992(429659534891842853950179275957642701\alpha^2 \\ +
    299279229875227284830201161759246904\alpha +
    4934717748383394121661044694765075727)}$}
& NO \\

$(107\alpha^2 + 72\alpha + 1345 : 42\alpha^2 + 16\alpha + 718 : 20\alpha^2 - 8\alpha
    + 420 : 328)$&
{$1/2(1594853540089340101\alpha^2 + 1110894791062576754\alpha + 18317182399622734527)$}
& NO \\

    \end{tabularx}
\end{table}

\begin{table}[h!]
    \centering
    \begin{tabularx}{0.82\paperheight}{>{\centering\arraybackslash}m{0.37\paperheight} >{\centering\arraybackslash}m{0.37\paperheight} >{\centering\arraybackslash}m{0.06\paperheight}}
    \multicolumn{3}{l}{Points with $\alpha$ satisfying $\alpha^3 - \alpha^2 + 53\alpha + 147 = 0$.} \\ [0.8ex]
    Coordinates & $j$-invariant & CM   \\ [0.5ex] 
        \hline \hline \rule{0pt}{2.8ex}

$(3(47\alpha^2 - 38\alpha + 3271) : 31\alpha^2 - 250\alpha + 7331 : 3(13\alpha^2
    - 48\alpha + 1767) : 5286)$ &
{$\substack{1/134217728(5760195975601694894538881657752708\alpha^2 \\ -
    19601232799455021073603806896166691\alpha +
    352389719111433433199614166785341899)}$}
& NO \\

$(39\alpha^2 - 129\alpha + 2550 : 11\alpha^2 - 42\alpha + 927 : 4\alpha^2 - 2\alpha
    + 430 : 292)$ &
{$\substack{1/2(5908628533072\alpha^2 - 20106331622569\alpha + 361470319752701)}$}
& NO \\
    \end{tabularx}
\end{table}

\begin{table}[h!]
    \centering
    \begin{tabularx}{0.82\paperheight}{>{\centering\arraybackslash}m{0.37\paperheight} >{\centering\arraybackslash}m{0.37\paperheight} >{\centering\arraybackslash}m{0.06\paperheight} }
    \multicolumn{3}{l}{Points with $\alpha$ satisfying $\alpha^3 - 2\alpha - 3 =0$.} \\ [0.8ex]
    Coordinates & $j$-invariant & CM   \\ [0.5ex] 
        \hline \hline \rule{0pt}{2.8ex}

$(-15\alpha^2 - 27\alpha - 21 : 3\alpha^2 + 7\alpha + 9 : -\alpha^2 - \alpha + 1 : 4)$&
$-26\alpha^2 - 123\alpha - 74$
& NO \\

$(-6\alpha^2 + 81\alpha + 120 : -11\alpha^2 + 15\alpha + 220 : 18\alpha^2 + 24\alpha +
    174 : 267)$&
{$\substack{-135538657185216969134696061309629\alpha^2 \\ - 256613875330390280002820617300050\alpha -
    214766963414423266964158036517900}$}
& NO \\

    \end{tabularx}
\end{table}

\begin{table}[h!]
    \
    \begin{tabularx}{0.82\paperheight}{>{\centering\arraybackslash}m{0.37\paperheight} >{\centering\arraybackslash}m{0.37\paperheight} >{\centering\arraybackslash}m{0.06\paperheight}}
    \multicolumn{3}{l}{Points with $\alpha$ satisfying  $\alpha^3 - 3\alpha - 4=0$.} \\ [0.8ex]
    Coordinates & $j$-invariant & CM   \\ [0.5ex] 
        \hline \hline \rule{0pt}{2.8ex}

$(11\alpha^2 + 23\alpha + 20 : 3\alpha^2 + 5\alpha + 8 : \alpha^2 + \alpha + 4 : 6)$&
$117843669992\alpha^2 + 258763889291\alpha + 214668790692$
& NO \\

$(131\alpha^2 + 98\alpha + 80 : 38\alpha^2 + 80\alpha + 410 : \alpha^2 + 91\alpha +
    396 : 563)$&
{$\substack{29734407481985808940410283825504\alpha^2 + 65291506111938160265579873086771\alpha\\ +
    54165390933943542489231382553348}$}
& NO \\

    \end{tabularx}
\end{table}

\newpage
\subsubsection{$X_0(57)$}
Model for $X_0(57)$: 
\begin{align*}
&x_0x_2 - x_1^2 + 2x_1x_3 + 2x_1x_4 - 2x_2^2 - 2x_2x_3 + 3x_2x_4 - x_3^2 - 2x_3x_4 - x_4^2=0,\\
&x_0x_3 - x_1x_2 - 2x_1x_4 + 4x_2x_3 - 6x_2x_4 - x_3^2 + 5x_3x_4 - 5x_4^2=0,\\
&x_0x_4 - x_2^2 + x_2x_3 - 2x_2x_4 - 2x_4^2=0.
\end{align*}
Genus of $X_0(57)$: 5.\\
Cusps: $(1:0:0:0:0),(1:1:0:1:0),(3:3:1:2:1),(3:9/2:-1/2:7/2:1)$.\\
$X_0^+(57)$: hyperelliptic curve 
$y^2 = x^6 - 2x^5 + 3x^4 + 3x^2 - 2x + 1 $\\ 
Group structure of $J_0^+(57)(\Q)$: $  \Z \cdot[(1:1:0) - (1:-1:0)] \oplus \Z/3\Z \cdot [(0 : -1 : 1) - (1:-1:0)] $. \\
Group structure of $G\subset J_0(57)(\Q)$: $G= \Z\cdot D_1\oplus \Z/6\Z\cdot D_{\mathrm{tor},1}\oplus \Z/30\Z\cdot D_{\mathrm{tor},2}$, where $D_{\mathrm{tor},1}=[(1:1:0:1:0)-(1:0:0:0:0)]$, $D_{\mathrm{tor},2}=[(3:3:1:2:1)-(1:0:0:0:0)]$ and $D_1=[(1:1:0:1:0) + (3:9/2:-1/2:7/2:1) + - P+\overline{P}]$ where
\begin{align*}
P := \left(-\frac{1}{3}(\sqrt{-2} -1 ) + 1 : -\frac{4}{3}(\sqrt{-2} -1 ) + 1 : \frac{1}{3}(\sqrt{-2} -1 ) : -\frac{2}{3}(\sqrt{-2} -1 ) + 1 : 1\right).
\end{align*}
satisfying $\rho(P)=(1:-1:0)$.\\
Primes used in sieve: 17, 43.

\begin{table}[h!]
    \centering
    \begin{tabularx}{0.82\paperheight}{>{\centering\arraybackslash}m{0.37\paperheight} >{\centering\arraybackslash}m{0.37\paperheight} >{\centering\arraybackslash}m{0.06\paperheight} }
    \multicolumn{3}{l}{Points with $\alpha$ satisfying $\alpha^3 -\alpha^2 + \alpha +2 = 0$.} \\ [0.8ex]
     Coordinates & $j$-invariant & CM   \\ [0.5ex] 
        \hline \hline \rule{0pt}{2.8ex}

$(\alpha^2 - 3\alpha + 7 : 2(-\alpha^2 + \alpha + 1) : \alpha^2 - \alpha + 1 : -\alpha^2 + \alpha + 3 : 2)$&
$-9209701\alpha^2 + 16880722\alpha - 21658080$&
NO \\
$(2\alpha^2 - 3\alpha + 8 : 2\alpha^2 - 3\alpha + 11 : -\alpha^2 - 1 : 2\alpha^2 -
    3\alpha + 11 : 3)$ &
$-1580665385097\alpha^2 - 921044021830\alpha + 291907686304$&
NO \\
$(-2\alpha^2 + 4 : \alpha^2 - 3\alpha + 1 : \alpha^2 - 1 : \alpha^2 + 3 : 1)$&
$114\alpha^2 - 163\alpha - 218$&
NO \\
$(\alpha^2 - 3\alpha + 7 : 4(\alpha^2 + 1 ): \alpha^2 - 3\alpha - 1 : \alpha^2 - 3\alpha + 3 : 4)$ &
$-18553245055702\alpha^2 + 33591312773593\alpha - 45780203714514$ &
NO \\
    \end{tabularx}
\end{table}

\begin{table}[h!]
    \centering
    
    \begin{tabularx}{0.82\paperheight}{>{\centering\arraybackslash}m{0.37\paperheight} >{\centering\arraybackslash}m{0.37\paperheight} >{\centering\arraybackslash}m{0.06\paperheight} }
    \multicolumn{3}{l}{Points with $\alpha$ satisfying $\alpha^3 - \alpha^2 -2=0$.} \\ [0.8ex]
    Coordinates & $j$-invariant & CM   \\ [0.5ex] 
        \hline \hline \rule{0pt}{2.8ex}
$(3\alpha + 9 : -3\alpha^2 + 9\alpha + 12 : -2 : 2(3\alpha + 4) : 4)$ &
$1/2(-7449\alpha^2 - 1673\alpha - 948)$&
NO \\
$(2\alpha + 2 : -\alpha^2 + \alpha + 4 : \alpha^2 - \alpha : -\alpha^2 + \alpha + 4 : 2)$&
$1/4(-1810169986337\alpha^2 - 1259199868019\alpha - 2135110941902)$&
NO \\
$(3\alpha^2 - 2\alpha + 6 : 4(\alpha^2 - \alpha + 3) : -\alpha^2 + 2\alpha - 2 : 3\alpha^2 - 2\alpha +
    10 : 4)$&
$1/536870912(-17777315457\alpha^2 + 14416014175\alpha - 86766426948)$&
NO \\
$(3\alpha + 3 : -3\alpha^2 + 9 : 2 : -3\alpha + 7 : 2)$&
$1/1024(-73465339921\alpha^2 - 51104061891\alpha - 86652910958)$&
NO \\
    \end{tabularx}
\end{table}

\begin{table}[h!]
    \centering
    \begin{tabularx}{0.82\paperheight}{>{\centering\arraybackslash}m{0.37\paperheight} >{\centering\arraybackslash}m{0.37\paperheight} >{\centering\arraybackslash}m{0.06\paperheight} }
    \multicolumn{3}{l}{Points with $\alpha$ satisfying $\alpha^3 -16 \alpha -27=0$.} \\ [0.8ex]
    Coordinates & $j$-invariant & CM   \\ [0.5ex] 
        \hline \hline \rule{0pt}{2.8ex}
$(25\alpha^2 - 53\alpha - 243 : 4\alpha^2 - 2\alpha : -\alpha^2 + 5\alpha - 9 :
17\alpha^2 - 31\alpha - 153 : 18)$&
{$\substack{1/1570042899082081611640534563 (-38600674637161069583675311194112\alpha^2 \\ +
    124697130934598013699081937354752\alpha \\ + 256826826563492878403632537894912)}$}&
        NO \\
$(55\alpha^2 - 27\alpha + 983 : 24(2\alpha^2 + 76) : 9(-5\alpha^2 + 9\alpha + 35) :
27(\alpha^2 + 3\alpha + 65) : 648)$ &
$1/3(353435648\alpha^2 - 1133445120\alpha - 2409267200)$ &
NO \\
$(-8\alpha^2 + 20\alpha + 90 : \alpha^2 - 4\alpha : \alpha^2 - \alpha - 15 : -2\alpha^2 +
    5\alpha + 24 : 3)$ &
$1/27(-29985636352\alpha^2 - 139955503104\alpha - 173460488192)$ &
NO \\
$(5\alpha^2 + 45\alpha + 604 : -7\alpha^2 - 63\alpha + 661 : 9(3\alpha^2 - 4\alpha -
    22) : -4\alpha^2 - 36\alpha + 577 : 279)$ &
{$\substack{1/1162261467(-1754586552521644333563904\alpha^2\\ + 3777138616478919143030784\alpha +
    20593798011121996168462336)}$} &
NO \\

    \end{tabularx}
\end{table}

\newpage
\subsubsection{$X_0(61)$}
Model for $X_0(61)$: 
\begin{align*}
& x_0^2  x_3 + x_0  x_1  x_3 - 2  x_0  x_3^2 - 2  x_1^3 - 6  x_1^2  x_2 + 5  x_1^2  x_3 - 5  x_1  x_2^2 + 4  x_1  x_2  x_3 - 6  x_2^3 + 14  x_2^2  x_3 - 11  x_2  x_3^2 + 4  x_3^3=0,\\
& x_0  x_2 - x_1^2 - x_1  x_2 - 2  x_2^2 + 2  x_2  x_3 - x_3^2=0
\end{align*}
Genus of $X_0(61)$: 4.\\
Cusps: $(1:0:0:0)$, $(1:0:1:1)$.\\
$X_0^+(61)$: elliptic curve $y^2 + xy = x^3 + 6x^2 + 11x + 6$ of conductor 61.\\
Group structure of $J_0^+(61)(\Q)$: $ \Z\cdot [Q-O]$, where $Q=(-1:1:1)$.\\
Group structure of $G\subset J_0(61)(\Q)$: $G= \Z\cdot D_1\oplus \Z/5\Z\cdot D_{\mathrm{tor}}$, where $D_{\mathrm{tor}}=[(1:0:1:1)-(1:0:0:0)]$ and $D_1=[P+\overline{P}-(1:1:1:1)-(1:0:1:1)]=\rho^*[Q-O]$ for 
\[
P:=\left(0: \frac12(\sqrt{-3} - 1): 1: 1\right)\in X_0(61)(\Q(\sqrt{-3}))
\]
satisfying $\rho(P)=Q$.\\
Primes used in sieve: 31, 19, 53, 23.

\begin{table}[h!]
    \centering
    \begin{tabularx}{0.82\paperheight}{>{\centering\arraybackslash}m{0.37\paperheight} >{\centering\arraybackslash}m{0.37\paperheight} >{\centering\arraybackslash}m{0.06\paperheight} }
    \multicolumn{3}{l}{Points with $\alpha$ satisfying  $\alpha^3 - 2\alpha - 20 =0$.} \\ [0.8ex]
    Coordinates & $j$-invariant & CM   \\ [0.5ex] 
        \hline \hline \rule{0pt}{2.8ex}

$(4\alpha^2 + 16\alpha - 16 : -6\alpha^2 + 12\alpha - 12 : 5\alpha^2 + 2\alpha + 34 : 72)$&
{$\substack{-1/2147483648(1996874400742389760403631439516303\alpha^2 \\+
    5909569160562647305266609005744561\alpha + 13495086575687455058014736453622810)}$}
& NO \\

$(\alpha^2 + 2\alpha + 8 : -2\alpha - 4 : 4 : 0)$ &
$1/2(148823\alpha^2 - 301199\alpha - 470790)$
& NO \\
    \end{tabularx}
\end{table}
\newpage
\subsubsection{$X_0(65)$}
\label{subsubsection_65cubic}
Model for $X_0(65)$:
\begin{align*}
&x_0 x_2 - x_1^2 + x_1 x_4 - 2 x_2^2 - x_2 x_3 + 3 x_2 x_4 - x_3^2 + 2 x_3 x_4 - 2 x_4^2=0,\\
&x_0 x_3 - x_1 x_2 - 2 x_2^2 - x_2 x_3 + 4 x_2 x_4 - x_3^2 + 2 x_3 x_4 - 2 x_4^2=0,\\
&x_0 x_4 - x_1 x_3 - 2 x_2^2 - 3 x_2 x_3 + 5 x_2 x_4 + 3 x_3 x_4 - 3 x_4^2=0.
\end{align*}
Genus of $X_0(65)$: 5.\\
Cusps: $(1:0:0:0:0),(1:1:1:1:1),(1/3:2/3:2/3:2/3:1),(1/2:1/2:1/2:1/2:1)$.\\
$X_0^+(65)$: elliptic curve $y^2 + xy = x^3 - x$ of conductor 65.\\
Group structure of $J_0^+(65)(\Q)$: $\Z\cdot [Q-O]\oplus \Z/2\Z\cdot [(0:0:1)-O]$, where $Q=(1:0:1)$.\\
Group structure of $J_0(65)(\Q)$: $\Z\cdot D_1 \oplus \Z/2\Z\cdot (-9D_{\mathrm{tor},1}+2D_{\mathrm{tor},2})\oplus \Z/84\Z\cdot (17D_{\mathrm{tor},1}+13D_{\mathrm{tor},2})$, where $D_{\mathrm{tor},1}=[(1:1:1:1:1)-(1:0:0:0:0)]$, $D_{\mathrm{tor},2}=[(1/3:2/3:2/3:2/3:1)-(1:0:0:0:0)]$ and $D_1=[P+\overline{P}-(1:0:0:0:)-(1:1:1:1:1)]=\rho^*([Q-O])$ for
\[
P=\left(0:1:\frac12(1+i):1:1\right)\in X_0(65)(\Q(i))
\]
satisfying $\rho(P)=Q$. \\
Primes used in sieve: 17, 23.

\begin{table}[h!]
    \centering
    \begin{tabularx}{0.82\paperheight}{>{\centering\arraybackslash}m{0.37\paperheight} >{\centering\arraybackslash}m{0.37\paperheight} >{\centering\arraybackslash}m{0.06\paperheight}}
    \multicolumn{3}{l}{Points with $\alpha$ satisfying $\alpha^3 - \alpha^2 + \alpha -2=0$.} \\ [0.8ex]
 Coordinates & $j$-invariant & CM   \\ [0.5ex] 
        \hline \hline \rule{0pt}{2.8ex}
        $(-3\alpha^2 + 15\alpha + 23 : 9\alpha^2 - 2\alpha + 17 : 10\alpha^2 - 7\alpha + 38
: 7\alpha^2 + 8\alpha + 18 : 43)$&
{$\substack{-30997852136030836533752464797632\alpha^2 - 10948750243223404067192473365696\alpha \\ -
    45813810060690523791258917233216}$}
& NO \\
$(28\alpha^2 + 10\alpha + 42 : -9\alpha^2 - 3\alpha - 12 : 3\alpha^2 + \alpha + 6 : -\alpha^2 :
2)$&
$16448\alpha^2 + 49984\alpha - 102976$
& NO \\

$(16(\alpha^2 - 3\alpha + 3) : 16(\alpha^2 - \alpha + 1) : 3\alpha^2 + 7\alpha + 5 : 2(-\alpha^2 +
    3\alpha + 9) : 32)$&
$1629248\alpha^2 - 2904256\alpha + 946624$
& NO \\

$(4(3\alpha^2 - 5\alpha + 14) : 2(2\alpha^2 - 11\alpha + 40) : -5\alpha^2 - 7\alpha + 61 :
2(-4\alpha^2 - \alpha + 35) : 92)$&
{$136553952252949568\alpha^2 + 102819038189152064\alpha - 389190231909361216$}
& NO \\
    \end{tabularx}
\end{table}

\newpage
\subsubsection{$X_0(67)$}
Model for $X_0(67)$:
\begin{align*}
   &x_0 x_2 - x_1^2 + 2 x_1 x_3 + 2 x_1 x_4 - 2 x_2^2 - 2 x_2 x_3 + 3 x_2 x_4 - x_3^2 - 2 x_3 x_4 - x_4^2=0,\\
&x_0 x_3 - x_1 x_2 - 2 x_1 x_4 + 4 x_2 x_3 - 6 x_2 x_4 - x_3^2 + 5 x_3 x_4 - 5 x_4^2=0,\\
&x_0 x_4 - x_2^2 + x_2 x_3 - 2 x_2 x_4 - 2 x_4^2=0.
\end{align*}
Genus of $X_0(67)$: 5.\\
Cusps: $(1:0:0:0:0),(1/2 : 1 : 1/2 : 1/2 : 1)$.\\
$X_0^+(67)$: genus 2 hyperelliptic curve $y^2 = x^6 - 2x^5 + x^4 + 2x^3 + 2x^2 + 4x + 1$.\\
Group Structure of $J_0^+(67)(\Q)$: $ \Z\cdot [Q_1-(1:-1:0)]\oplus \Z[Q_2-(1:-1:0)]$, where $Q_1=(1:1:0)$ and $Q_2=(0:1:1)$. \\
Group Structure of $G\subset J_0(67)(\Q)$: $G = \Z\cdot D_1\oplus \Z D_2\oplus \Z/11\Z\cdot D_{\mathrm{tor}}$, where $D_{\mathrm{tor}}:=[(1/2 : 1 : 1/2 : 1/2 :1)-(1:0:0:0:0)]$, $D_1:=[P_2+\overline{P}_2-(1:0:0:0:0)-(1/2 : 1 : 1/2 : 1/2 :1)]$ and $D_2:=[P_1+\overline{P}_1-(1:0:0:0:0)-(1/2 : 1 : 1/2 : 1/2 :1)]$ for $P_1,P_2$ as below satisfying $\rho(P_1)=Q_1$ and $\rho(P_2)=Q_2$.\\
Primes used in sieve: $73,59,53,37,17,131,167,359$.

\begin{align*}
    & P_1= \left( \frac{1}{18}(- \sqrt{-2} + 4): \frac{1}{18}( \sqrt{-2} + 14): \frac{1}{18} (- \sqrt{-2} + 4): \frac19 (- \sqrt{-2} + 4): 1 \right) \\
& P_2 = \left( 0:  \frac{1}{11} ( \sqrt{-11} + 11):  \frac{1}{22} (- \sqrt{-11} + 11):  \frac{1}{22} ( \sqrt{-11} + 11): 1 \right) \\
\end{align*}

\begin{table}[h!]
    \centering
    \begin{tabularx}{0.82\paperheight}{ >{\centering\arraybackslash}m{0.37\paperheight} >{\centering\arraybackslash}m{0.37\paperheight} >{\centering\arraybackslash}m{0.06\paperheight} }
    \multicolumn{3}{l}{Points with $\alpha$ satisfying $\alpha^3 - \alpha^2 + \alpha -3=0$.} \\ [0.8ex]
      Coordinates & $j$-invariant & CM   \\ [0.5ex] 
        \hline \hline \rule{0pt}{2.8ex}
                 $(7\alpha^2 - 4\alpha + 13 : -\alpha^2 - 2\alpha + 11 : \alpha^2 + 2\alpha + 7 :
-\alpha^2 - 2\alpha + 11 : 18)$ & $-8522452\alpha^2 - 4897575\alpha - 16236160$ & NO\\

            $(19\alpha^2 - 13\alpha + 79 : -2\alpha^2 - 20\alpha + 184 : 19\alpha^2 - 13\alpha
    + 79 :-\alpha^2 - 10\alpha + 92 : 203)$    & {$-3728415892268407\alpha^2 - 2142881183120745\alpha - 7102903223125105 $} & NO\\
      \end{tabularx}
\end{table} 

\begin{table}[h!]
    \centering
    \begin{tabularx}{0.82\paperheight}{ >{\centering\arraybackslash}m{0.37\paperheight} >{\centering\arraybackslash}m{0.37\paperheight} >{\centering\arraybackslash}m{0.06\paperheight} }
    \multicolumn{3}{l}{Points with $\alpha$ satisfying $\alpha^3 - \alpha^2 -3 \alpha +5 =0$.}\\ [0.5ex]
      Coordinates & $j$-invariant & CM   \\ [0.5ex] 
        \hline \hline \rule{0pt}{2.8ex}
        $(11(3\alpha^2 - 18\alpha + 29) : 7\alpha^2 - 94\alpha + 371 : 13(\alpha^2 - 4\alpha + 9) :
2(-\alpha^2 - 7\alpha + 90) : 286)$ & {$1821644613377781192000\alpha^2 - 5318545688116020360000\alpha + 4744756895498802918000$ } &  $-268$ \\
    \end{tabularx}
\end{table}
\newpage
\subsubsection{$X_0(73)$}
Model for $X_0(73)$:
\begin{align*}
&x_0 x_2 - 2 x_1^2 + 2 x_1 x_2 - 2 x_1 x_4 - x_2^2 + 3 x_2 x_3 + 3 x_3^2 - x_4^2=0,\\
&x_0 x_3 - 1/2 x_1 x_2 - x_1 x_3 + 1/2 x_2^2 - 1/2 x_2 x_3 + x_2 x_4 - 4 x_3^2 + 9/2 x_3 x_4 - 1/2 x_4^2=0,\\
&x_0 x_4 - x_1 x_3 + x_1 x_4 - x_2 x_3 - 5 x_3^2 + 4 x_3 x_4=0.
\end{align*}
Genus of $X_0(73)$: 5.\\
Cusps: $(1:0:0:0:0),(1:1:1:0:0)$.\\
$X_0^+(73)$: genus 2 hyperelliptic curve $y^2 = x^6 + 2x^5 + x^4 + 6x^3 + 2x^2 - 4x + 1$.\\
Group Structure of $J_0^+(73)(\Q)$: $\Z\cdot [Q_1-(1:1:0)]\oplus \Z\cdot [Q_2-(1:1:0)]$, where $Q_1:=(0:-1:1)$ and $Q_2:=(0:1:1)$. \\
Group Structure of $G\subset J_0(73)(\Q)$: $G= \Z\cdot D_1\oplus \Z\cdot D_2 \oplus \Z/6\Z\cdot D_{\mathrm{tor}}$, where $D_{\mathrm{tor}}=[(1:0:0:0:0)-(1:1:1:0:0)]$, $D_1=[P_1+\overline{P}_1-(1:0:0:0:0)-(1:1:1:0:0)]$, $D_2=[P_2+\overline{P}_2-(1:0:0:0:0)-(1:1:1:0:0)]$ for $P_1$, $P_2$ as below and satisfying $\rho(P_1)=Q_1$ and $\rho(P_2)=Q_2$. \\
Primes used in sieve: $19,41,59,43,13,103$.

\begin{align*}
    & P_1 = \left( 1/7 (\sqrt{-19} - 10): \frac{1}{14} (\sqrt{-19} - 17): \frac{1}{14} (\sqrt{-19} - 17): \frac{1}{14} (\sqrt{-19} + 11): 1 \right) \\
    & P_2 = \left( \frac{1}{6} (-\sqrt{-2} - 8): -1: \frac{1}{6} (-\sqrt{-2} - 8): \frac{1}{6} (-\sqrt{-2} + 4): 1 \right)\\
\end{align*}

\begin{table}[h!]
    \centering
    \begin{tabularx}{0.82\paperheight}{ >{\centering\arraybackslash}m{0.37\paperheight} >{\centering\arraybackslash}m{0.37\paperheight} >{\centering\arraybackslash}m{0.06\paperheight} }
    \multicolumn{3}{l}{Points with $\alpha$ satisfying $\alpha^3 - \alpha^2 +7 \alpha+8=0$.} \\ [0.8ex]
 Coordinates & $j$-invariant & CM   \\ [0.5ex] 
        \hline \hline \rule{0pt}{2.8ex}
        $ (-\alpha^2 + 3\alpha - 1 : -\alpha - 2 : \alpha - 3 : 1 : 0)$ &
$-494084542528\alpha^2 + 945810849040\alpha - 4323317577856$ & NO \\

        $(5\alpha^2 - 9\alpha + 39 : \alpha^2 - \alpha + 4 : 4\alpha^2 - 7\alpha + 29 : -\alpha^2
    + 2\alpha - 7 : 2)$ &
{$\substack{118037570988753730717632\alpha^2 - 43855988354612266206411760\alpha\\ -
    40194848013188069677969536}$} & NO \\
    \end{tabularx}
\end{table}

\newpage
\subsection{Quartic Points}
\label{subsection_quarticpoints}
Recall that $X_0(57)$ and $X_0(65)$ both have genus 5 and the Mordell--Weil groups of their jacobians over $\Q$ have rank one.
This puts them within the bounds to apply Theorem \ref{bigtheorem}.
However, there is a problem.
These curves have infinitely many degree four maps to $\P^1$!
For $X_0(65)$, we are in luck: all of the maps defined over $\Q$ factor through the quotient $X_0^+(65)$.
On the other hand, $X_0(57)$ is not bielliptic, making us unable to determine its isolated quartic points.

Using the same model as in \S \ref{subsubsection_65cubic} we now list a representative for each of the 44 Galois conjugacy classes of quartic points on $X_0(65)$ which do not come from the quotient $X_0^+(65)$.
As before, the points which give rise to $\Q$-curves are exactly the CM points.

\begin{table}[h!]    
\centering  
\begin{tabularx}{0.82\paperheight}{>{\centering\arraybackslash}m{0.37\paperheight} >{\centering\arraybackslash}m{0.37\paperheight} >{\centering\arraybackslash}m{0.06\paperheight} } 
\multicolumn{3}{l}{Points with $\alpha$ satisfying $\alpha^4 + 3\alpha^2 + 1 = 0$. } \\ [0.8ex]
Coordinates & $j$-invariant & CM   \\ [0.5ex]  
\hline \hline \rule{0pt}{2.8ex}
$(-2\alpha^3 + \alpha^2 - \alpha + 1 : \alpha^3 + \alpha^2 + 1 : -\alpha^2 : \alpha^2 + 1 : 1)$ &
$1728$ 
& $-4$ \\
$(2\alpha^3 + 2\alpha^2 + 4\alpha + 12 : 4\alpha^3 - \alpha^2 + 9\alpha + 11 :
-2\alpha^3 + 3\alpha^2 - 5\alpha + 19 : 3\alpha^3 + 2\alpha^2 + 8\alpha + 19 :
18)$ &
$19691491018752\alpha^2 + 51552986141376$
& $-100$ \\
    \end{tabularx}
\end{table}

 \begin{table}[h!]     \centering     
 \begin{tabularx}{0.82\paperheight}{>{\centering\arraybackslash}m{0.37\paperheight} >{\centering\arraybackslash}m{0.37\paperheight} >{\centering\arraybackslash}m{0.06\paperheight} }
 \multicolumn{3}{l}{Points with $\alpha$ satisfying $\alpha^4 - \alpha^3 + 2\alpha^2 + \alpha + 1 =0$.} \\ [0.8ex]
 Coordinates & $j$-invariant & CM   \\ [0.5ex]          \hline \hline \rule{0pt}{2.8ex}
\small{$(8\alpha^3 + 18\alpha^2 - 46\alpha + 220 : 93\alpha^3 - 92\alpha^2 + 128\alpha
    + 509 : -35\alpha^3 + 102\alpha^2 - 100\alpha + 363 : -77\alpha^3 +
    128\alpha^2 - 220\alpha + 413 : 482)$} &
\small{$-146329141248\alpha^3 - 619860197376$}
& $-75$ \\
$(2\alpha^3 - 4\alpha^2 + 6\alpha : \alpha^3 + \alpha + 6 : -\alpha^3 + 2\alpha^2 -
    3\alpha + 2 :-\alpha^3 + \alpha^2 - \alpha + 1 : 4)$ &
$0$
& $-3$ \\
    \end{tabularx}
\end{table}

 \begin{table}[h!]     \centering     
 \begin{tabularx}{0.82\paperheight}{>{\centering\arraybackslash}m{0.37\paperheight} >{\centering\arraybackslash}m{0.37\paperheight} >{\centering\arraybackslash}m{0.06\paperheight}} 
 \multicolumn{3}{l}{Points with $\alpha$ satisfying $\alpha^4 - 2\alpha^3 - 4\alpha^2 + 4\alpha + 10 =0$.} \\ [0.8ex]
 Coordinates & $j$-invariant & CM   \\ [0.5ex]          \hline \hline \rule{0pt}{2.8ex}
\footnotesize{$(3(9\alpha^3 + 3\alpha^2 - 29\alpha + 65) : 5(4\alpha^3 + 11\alpha^2 - 29\alpha +
    74) : 23\alpha^3 - 89\alpha^2 - 58\alpha + 730 : -7\alpha^3 +
    46\alpha^2 - 58\alpha + 610 : 870)$} &
$1/6(-11331665\alpha^3 + 42997532\alpha^2 - 23611406\alpha - 39687952)$
& NO  \\
\small{$(7\alpha^3 - 15\alpha^2 - 21\alpha + 65 : -6\alpha^3 + 25\alpha^2 + \alpha - 29
: \alpha^3 - 7\alpha^2 + 14\alpha + 19 : -5\alpha^3 + 18\alpha^2 - 2\alpha + 7
: 51)$} &
{$1/393216(565729615\alpha^3 + 1124273732\alpha^2 + 11474554\alpha - 508128592)$}
& NO   \\
$(3(-3\alpha^3 - 2\alpha^2 + 8\alpha + 12) : 3(3\alpha^3 - 12\alpha - 12) :
3(-\alpha^3 + 4\alpha + 6) : -5\alpha^3 - 4\alpha^2 + 10\alpha + 20 : 6) $ &
{$1/48(-8144181273275\alpha^3 + 25312371833066\alpha^2  + 8043694506022\alpha -
    57216807377536)$}
& NO  \\
$(-2\alpha^3 + 5\alpha^2 + \alpha + 5 : 3\alpha^2 - 6 : \alpha^3 - \alpha^2 -
    5\alpha + 8 : -\alpha^3 + 4\alpha^2 - 4\alpha + 4 : 9)$ &
$1/4(103767807\alpha^3 - 312011900\alpha^2 - 149464130\alpha + 782896654)$
& NO   \\
    \end{tabularx}
\end{table}

\begin{table}[h!]     \centering     
\begin{tabularx}{0.82\paperheight}{>{\centering\arraybackslash}m{0.37\paperheight} >{\centering\arraybackslash}m{0.37\paperheight} >{\centering\arraybackslash}m{0.06\paperheight}} 
\multicolumn{3}{l}{Points with $\alpha$ satisfying $\alpha^4 - 5\alpha^2 + 36 =0$.} \\ [0.8ex]
Coordinates & $j$-invariant & CM   \\ [0.5ex]          \hline \hline \rule{0pt}{2.8ex}
\small{$(8\alpha^3 - 8\alpha^2 + 32\alpha + 160 : \alpha^3 + 3\alpha^2 + 4\alpha + 84 :
6(\alpha^3 - 5\alpha + 54) : 6(\alpha^3 - 3\alpha^2 + 4\alpha + 60) : 144)$} &
\small{$1/16384(188921499633\alpha^3 + 905537632545\alpha^2 + 1655605513116\alpha +
    766437509340)$}
& NO   \\
{$(-10\alpha^3 + 12\alpha^2 - 286\alpha - 300 : 12(2\alpha^3 + 11\alpha^2 + 17\alpha
    + 60) : 5\alpha^3 - 6\alpha^2 + 143\alpha + 954 : 4(16\alpha^3 +
    21\alpha^2 - 65\alpha + 78) : 1608)$} &
{$\substack{1/73786976294838206464(427634371581028354195137\alpha^3 -
    438836098747959200559375\alpha^2 \\ - 3302486685333702076993956\alpha +
    9869446094702964951347100)}$}
& NO  \\
    \end{tabularx}
\end{table}

\begin{table}[h!]     \centering     
\begin{tabularx}{0.82\paperheight}{>{\centering\arraybackslash}m{0.37\paperheight} >{\centering\arraybackslash}m{0.37\paperheight} >{\centering\arraybackslash}m{0.06\paperheight}}
\multicolumn{3}{l}{Points with $\alpha$ satisfying $\alpha^4 -\alpha^2 + 64 = 0$.} \\ [0.8ex]
Coordinates & $j$-invariant & CM   \\ [0.5ex]          \hline \hline \rule{0pt}{2.8ex}
{$(-23\alpha^3 + 56\alpha^2 - 41\alpha - 184 : 2(3\alpha^3 - 24\alpha^2 +
    61\alpha + 24) : 4(\alpha^3 - 9\alpha + 72) : 8(-\alpha^3 + 9\alpha - 8) : 256)$} &
$1/512(-352692275\alpha^3 - 703166360\alpha^2 + 274492275\alpha + 6008246680)$
& NO  \\
{$(16(-\alpha^3 - 7\alpha + 72) : 16(-\alpha^3 - 7\alpha + 72) : -15\alpha^3 +
    8\alpha^2 + 207\alpha + 1400 : 2(-11\alpha^3 - 24\alpha^2 + 11\alpha + 856) :
    2048)$} &
\small{$\substack{1/73786976294838206464(185618461058956310814189183475\alpha^3 +
    451625115647050834040533476755\alpha^2 \\ + 191531841188243058702903432000\alpha -
    2509595271373538718782849717440)}$}
& NO   \\
    \end{tabularx}
\end{table}

\begin{table}[h!]     \centering     
\begin{tabularx}{0.82\paperheight}{>{\centering\arraybackslash}m{0.37\paperheight} >{\centering\arraybackslash}m{0.37\paperheight} >{\centering\arraybackslash}m{0.06\paperheight}} 
\multicolumn{3}{l}{Points with $\alpha$ satisfying $\alpha^4 + 10\alpha^2 - 13\alpha + 6 =0$.} \\ [0.8ex]
Coordinates & $j$-invariant & CM   \\ [0.5ex]          \hline \hline \rule{0pt}{2.8ex}
{$(4(3\alpha^3 + 11\alpha^2 - 57\alpha + 134) : 4(6\alpha^3 + 22\alpha^2 + 77\alpha
    + 77) :  -31\alpha^3 - 50\alpha^2 - 366\alpha + 971 : 2(25\alpha^3 +
    28\alpha^2 + 98\alpha + 289) : 764)$} &
{$\substack{1/329633646748081198527153512570976(-93996324565815115602099408132689707683077969\alpha^3 \\ - 27272583141368739953606302227340808768478213\alpha^2 \\ -
    922602504649377247568643703363209592787778419\alpha \\ +
    985686437922287224950855418754958911083274126)}$} 
& NO   \\
{$(-116\alpha^3 - 61\alpha^2 - 1021\alpha + 2589 : 3(-43\alpha^3 - 14\alpha^2
    - 341\alpha + 1479) :  9(4\alpha^3 + 8\alpha^2 + 47\alpha + 294) :
4617/171(2\alpha^3 + \alpha^2 + 22\alpha + 114) : 4617)$} &
\small{$\substack{1/8965117619822842085376(-44487012838602418265083128833\alpha^3 -
    37355632917799034954498764917\alpha^2 \\ - 466966355262376386084467179907\alpha +
    197644133686780526062866403118)}$}
& NO  \\
\scriptsize{$(-\alpha^3 + 3\alpha^2 + \alpha + 102 : 2(-4\alpha^3 - 2\alpha^2 - 37\alpha + 78)
: 2(-\alpha^3 - \alpha^2 - 12\alpha + 36) : 2(-\alpha^2 - \alpha + 42) : 144)$} &
{$\substack{1/20602102921755074907947094535686(394721925585506038939231516050863\alpha^3 \\ -
    2768103066312680500698037612142213\alpha^2 +
    3322177705362944338498999061316877\alpha \\ - 951585345694054611469818097448946)}$}
& NO  \\
$(2(-13\alpha^3 - 15\alpha^2 - 139\alpha + 46) : 6(2\alpha^3 + 3\alpha^2 + 20\alpha +
    4) : 9(\alpha^3 + 10\alpha - 13) : 54 : 0)$ &
{$\substack{1/1990656(-74128331233\alpha^3 - 1035513003989\alpha^2 + 1305535708381\alpha -
    583073890962)}$}
& NO  \\
    \end{tabularx}
\end{table}

\begin{table}[h!] 
\centering
\begin{tabularx}{0.82\paperheight}{>{\centering\arraybackslash}m{0.37\paperheight} >{\centering\arraybackslash}m{0.37\paperheight} >{\centering\arraybackslash}m{0.06\paperheight}}
\multicolumn{3}{l}{Points with $\alpha$ satisfying $\alpha^4 - 2\alpha^3 + 33\alpha^2 + 124\alpha + 324 =0$.} \\ [0.8ex]
Coordinates & $j$-invariant & CM   \\ [0.5ex]          \hline \hline \rule{0pt}{2.8ex}
{$\substack{(\alpha^3 + 79\alpha^2 - 696\alpha + 1420 : 3(-\alpha^3 + 29\alpha^2 - 60\alpha
    + 1388) : \\ 9(\alpha^3 + 7\alpha^2 + 24\alpha + 844) : 27(-\alpha^3 + 5\alpha^2 -
    36\alpha + 188) : 7776)}$} &
{$\substack{1/7962624(-43105827974237\alpha^3 - 309174638866315\alpha^2 \\ - 858567581971168\alpha -
    1476381721011948)}$}
& NO  \\
{$\substack{(21384/1458(7\alpha^3 - 22\alpha^2 + 233\alpha + 972) : 29\alpha^3 - 176\alpha^2 +
    1837\alpha + 11934 : \\ 6(-8\alpha^3 + 11\alpha^2 - 121\alpha + 1098) :
9(-\alpha^3 + 22\alpha^2 - 11\alpha + 1614) : 21384)}$} &
{$\substack{1/35860470479291368341504(12618031044935721102344083\alpha^3 -
    64157013721606511003683835\alpha^2 \\ + 555106151674649210688259872\alpha +
    153918822709072746459229012)}$} 
& NO  \\
{$\substack{(1053\alpha^3 - 997\alpha^2 + 36528\alpha + 219596 : 10(127\alpha^3 -
    643\alpha^2 + 2948\alpha + 41036) : \\ 4(33\alpha^3 - 677\alpha^2 + 868\alpha +
    81236) : 8(114\alpha^3 - 131\alpha^2 - 681\alpha + 49558) : 518080)}$} &
\small{$\substack{1/82408411687020299631788378142744(-909371081634429814189874867193897304637\alpha^3 \\ + 12586954158315417996000859906005034036565\alpha^2 -
    83534381706452646404083681075846947229408\alpha \\ +
    344518838187266283365991746158951784244372)}$}
& NO  \\
{$\substack{(8(-25\alpha^3 + 85\alpha^2 - 944\alpha - 2168) : 4(39\alpha^3 - 230\alpha^2 +
    1609\alpha + 2486) : \\ 45\alpha^3 - 153\alpha^2 + 920\alpha + 15980 :
2(-17\alpha^3 - 137\alpha^2 + 410\alpha - 3656) : 15584)}$} &
\small{$\substack{1/1318534586992324794108614050283904(-11434805526353461934615216899479090166078637\alpha^3 + \\ 5228953638910671575533618290081115693190725\alpha^2  + 
    83338824414498902469112145034057200015322912\alpha \\  +
    325680039927910810332592971628693692224201812)}$}
& NO  \\
    \end{tabularx}
\end{table}

\begin{table}[h!]    
\centering     
\begin{tabularx}{0.82\paperheight}{>{\centering\arraybackslash}m{0.37\paperheight} >{\centering\arraybackslash}m{0.37\paperheight} >{\centering\arraybackslash}m{0.06\paperheight}} 
\multicolumn{3}{l}{Points with $\alpha$ satisfying $\alpha^4 - 2\alpha^3 + 2\alpha^2 + 2\alpha - 4 =0$.} \\ [0.8ex]
Coordinates & $j$-invariant & CM   \\ [0.5ex]          \hline \hline \rule{0pt}{2.8ex}
{$\substack{(-384\alpha^3 + 1189\alpha^2 - 2064\alpha + 1526 : 150\alpha^3 -
    445\alpha^2 + 744\alpha - 386 : \\ -55\alpha^3 + 177\alpha^2 - 306\alpha + 302 :
16\alpha^3 - 53\alpha^2 + 86\alpha + 54 : 166)}$} &
{$\substack{1/73786976294838206464(89811909733112525574421\alpha^3 -
    247403443120239759236772\alpha^2 + \\ 470785458314526794481074\alpha -
    342240832269779732491866)}$}
& NO  \\
{$(11(3\alpha^3 - 4\alpha^2 + 6\alpha + 18) : -39\alpha^3 + 52\alpha^2 - 14\alpha -
    58 : 27\alpha^3 - 36\alpha^2 + 30\alpha + 162 : -19\alpha^3 - 4\alpha^2
    + 18\alpha + 62 : 88)$} &
{$\substack{1/4194304(12716759153730866\alpha^3 - 27727513674538234\alpha^2 \\ -
    16708346390401773\alpha + 29993540624153550)}$}
& NO  \\
{$\substack{(241\alpha^3 + 876\alpha^2 - 3030\alpha + 10510 : 2(317\alpha^3 -
    436\alpha^2 + 2\alpha + 6390) : \\ 4(-215\alpha^3 + 604\alpha^2 - 862\alpha + 2654)
: 8(-75\alpha^3 + 116\alpha^2 - 206\alpha + 1494) : 16288)}$} &
$\substack{1/4(8754479010695388380472550951\alpha^3 - 26931525036400964083187224764\alpha^2 \\ +
    46495731438962280331948301537\alpha - 32535041725515706432242799260)}$
& NO  \\
$(\alpha^3 + 3\alpha^2 : \alpha^3 - 2\alpha^2 - 2\alpha + 4 : 2\alpha + 4 :
    2\alpha^2 : 4)$ &
$\substack{1/64(-121794908716379\alpha^3 + 95010271213628\alpha^2 \\ - 127685261383726\alpha -
    399355122270266)}$
& NO  \\
    \end{tabularx}
\end{table}

\begin{table}[h!] 
\centering 
\begin{tabularx}{0.82\paperheight}{>{\centering\arraybackslash}m{0.37\paperheight} >{\centering\arraybackslash}m{0.37\paperheight} >{\centering\arraybackslash}m{0.06\paperheight}} 
\multicolumn{3}{l}{Points with $\alpha$ satisfying $\alpha^4 - 2\alpha^3 - 6\alpha^2 + 3\alpha + 8 =0$.} \\ [0.8ex]
Coordinates & $j$-invariant & CM   \\ [0.5ex]          \hline \hline \rule{0pt}{2.8ex}
\footnotesize{$(27\alpha^3 - 78\alpha^2 - 66\alpha + 193 : 10(-\alpha^3 + 2\alpha^2 + 6\alpha - 3)
: 4(-3\alpha^3 + 12\alpha^2 + 4\alpha - 7) : 8(\alpha^3 - 4\alpha^2 + 2\alpha + 9) :
80)$} &
$\substack{1/2(-73415413682485\alpha^3 - 96687558634718\alpha^2 \\ + 119780562840206\alpha +
    177065002925201)}$
& NO  \\
{$(2(33\alpha^3 + 21\alpha^2 - 31\alpha + 101) : 2(-5\alpha^3 - 59\alpha^2 +
    14\alpha + 301) : 27\alpha^3 + 73\alpha^2 - 137\alpha + 278 :
2(20\alpha^3 - 71\alpha^2 - 56\alpha + 331) : 614)$} &
{$\substack{1/36893488147419103232(297793755303701986554311435\alpha^3 \\ -
    1401720256809888215691668318\alpha^2  - 279382961471190479874478834\alpha \\+
    2112213001641271144506379601)}$}
& NO  \\
{$(-452\alpha^3 + 348\alpha^2 + 3048\alpha + 2713 : 164\alpha^3 - 70\alpha^2
    - 1221\alpha - 946 : -84\alpha^3 + 57\alpha^2 + 569\alpha + 696 :
29\alpha^3 - 30\alpha^2 - 193\alpha + 90 : 289)$} &
$\substack{1/32(-10248101735925\alpha^3 + 25536131264162\alpha^2 \\ + 56923104477966\alpha -
    95768669529519)}$
& NO  \\
{$(-167\alpha^3 + 518\alpha^2 + 178\alpha + 383 : -171\alpha^3 +
    590\alpha^2 + 294\alpha - 241 : -73\alpha^3 + 70\alpha^2 + 562\alpha + 741 :
5\alpha^3 - 90\alpha^2 + 166\alpha + 1091 : 1244)$} &
$\substack{1/8192(8699697034245003878315770635\alpha^3 \\ - 5951864245642353783162260318\alpha^2 
    - 60029965678541791682581607154\alpha \\ - 52891570257608085379196076719)}$ 
& NO   \\
    \end{tabularx}
\end{table}

 \begin{table}[h!]
 \centering     
 \begin{tabularx}{0.82\paperheight}{>{\centering\arraybackslash}m{0.37\paperheight} >{\centering\arraybackslash}m{0.37\paperheight} >{\centering\arraybackslash}m{0.06\paperheight}}
 \multicolumn{3}{l}{Points with $\alpha$ satisfying $\alpha^4 -\alpha^3 - 5\alpha^2 - 12\alpha + 4 =0$.} \\ [0.8ex]
 Coordinates & $j$-invariant & CM   \\ [0.5ex]          \hline \hline \rule{0pt}{2.8ex}
$(\alpha^3 + 5\alpha^2 - 7\alpha + 10 : \alpha^3 - 3\alpha^2 + 9\alpha + 10 :
-\alpha^3 + 3\alpha^2 + 7\alpha + 6 : \alpha^3 - 3\alpha^2 + \alpha + 10 : 16)$ &
$\substack{1/32768(1699588421018410771117575\alpha^3 + 4085001869259831852583379\alpha^2 +\\
    5405460540196338811614783\alpha - 1997445389150221782346154)}$
& NO  \\
$\substack{(2(63\alpha^3 + 197\alpha^2 + 285\alpha + 548) : 67\alpha^3 + 103\alpha^2 +
    367\alpha + 2010 : \\ 4(13\alpha^3 + 30\alpha^2 - 69\alpha + 390) :
4(15\alpha^3 - 17\alpha^2 - 28\alpha + 450) : 2684)}$ &
$\substack{1/128(42407499986747925235047\alpha^3 + 101927451713689293321395\alpha^2 \\+
    134875046777126873505951\alpha - 49839516594937823319530) }$
& NO  \\
$(-9\alpha^3 + 3\alpha^2 + 47\alpha + 182 : 2(3\alpha^3 - \alpha^2 - 21\alpha - 18)
: 8(-\alpha^3 + \alpha^2 + 5\alpha + 16) : 32 : 64) $ & 
{$1/8(1212893975\alpha^3 - 854535645\alpha^2 - 6316920369\alpha - 16421016074)$}
& NO  \\
$\substack{(8(-89\alpha^3 + 75\alpha^2 + 407\alpha + 1630) : 32(12\alpha^3 - 7\alpha^2 -
    86\alpha + 51) : \\ -535\alpha^3 + 289\alpha^2 + 2957\alpha + 13446 :
2(13\alpha^3 - 123\alpha^2 - 47\alpha + 3310) : 8864)}$ &
$\substack{1/147573952589676412928(82530874763434407104988528897559\alpha^3 \\ -
    58389322418959943410157155889373\alpha^2 -
    428793171837557068985508233974065\alpha \\- 1118102278756799246262107049140746)}$
& NO   \\
    \end{tabularx}
\end{table}

 \begin{table}[h!] 
 \centering
 \begin{tabularx}{0.82\paperheight}{>{\centering\arraybackslash}m{0.37\paperheight} >{\centering\arraybackslash}m{0.37\paperheight} >{\centering\arraybackslash}m{0.06\paperheight} }    
 \multicolumn{3}{l}{Points with $\alpha$ satisfying $\alpha^4 -\alpha^3 +\alpha^2 + 15\alpha + 8 =0$.} \\ [0.8ex]
 Coordinates & $j$-invariant & CM   \\ [0.5ex]          \hline \hline \rule{0pt}{2.8ex}
$(4(\alpha^3 - 2\alpha^2 - 5\alpha + 6) : 8(\alpha^2 + \alpha + 1) : \alpha^3 -
    4\alpha^2 + 5\alpha + 40 : 2(\alpha^3 - 3\alpha + 8) : 24)$ &
$\substack{1/18786186952704(-529096150156485062644239434\alpha^3 +
    833547216099653533626822633\alpha^2 \\ - 1008733619395736267637671387\alpha -
    7355999499831154760170196264)}$
& NO   \\
$\substack{(-282\alpha^3 + 891\alpha^2 - 1847\alpha + 496 : 202\alpha^3 -
    487\alpha^2 + 1197\alpha + 3602 : \\ -259\alpha^3 + 686\alpha^2 - 1438\alpha +
    695 : -17\alpha^3 - 3\alpha^2 + 84\alpha + 2248 : 3554)}$ &
$\substack{1/711559752519106944(1455331341030054027299383378\alpha^3 -
    4035164875167567001867400433\alpha^2 +\\ 8609137673795812100808665275\alpha +
    6567184709951911955583590440)}$
& NO   \\
$(2(-2\alpha^3 + 6\alpha^2 - 11\alpha + 4) : 3(\alpha^3 - 2\alpha^2 + 5\alpha + 18) :
3(-\alpha^3 + 3\alpha^2 - 7\alpha + 5) : 36 : 54)$ &
$\substack{1/24(295009796578\alpha^3 - 818007129969\alpha^2 \\+ 1745183842219\alpha +
    1331261559976)}$
& NO   \\
\footnotesize{$(2(7\alpha^3 - 12\alpha^2 + 13\alpha + 112) : 4(-2\alpha^3 + 3\alpha^2 - 5\alpha -
    20) : 7\alpha^3 - 12\alpha^2 + 19\alpha + 136 : -2\alpha^3 - 2\alpha + 16 :
    48)$} &
$\substack{1/54(456884149846\alpha^3 - 292804658823\alpha^2 \\- 2750405339723\alpha -
    1410079750568)}$
& NO \\
    \end{tabularx}
\end{table}

 \begin{table}[h!] 
 \centering     
 \begin{tabularx}{0.82\paperheight}{>{\centering\arraybackslash}m{0.37\paperheight} >{\centering\arraybackslash}m{0.37\paperheight} >{\centering\arraybackslash}m{0.06\paperheight}}
 \multicolumn{3}{l}{Points with $\alpha$ satisfying $\alpha^4 + 22\alpha^2 - 36\alpha - 1587 =0$.} \\ [0.8ex]
 Coordinates & $j$-invariant & CM   \\ [0.5ex]          \hline \hline \rule{0pt}{2.8ex}
{$\substack{(-7423\alpha^3 + 56049\alpha^2 - 380713\alpha + 4990147 :
10(31\alpha^3 - 3409\alpha^2 + 9017\alpha + 473485) : \\ 4(-823\alpha^3 +
    9049\alpha^2 - 123993\alpha + 1840387) : 8(-499\alpha^3 - 2823\alpha^2 -
    19569\alpha + 788591) : 8416960)}$} &
\small{$\substack{1/10317616411527050688(551500224281215664829167280769459\alpha^3 \\ -
    2916026023363419895375635644340249\alpha^2 + \\
    27551326444124285184035446453362223\alpha -
    165530077365192468536119114277605449)}$}
& NO  \\
{$\substack{(2(1540\alpha^3 + 24771\alpha^2 + 279122\alpha + 2323657) :
5413\alpha^3 + 747\alpha^2 + 217423\alpha + 5494649 : \\
6063\alpha^3 + 18979\alpha^2 + 101933\alpha + 3263633 :
3317\alpha^3 + 23025\alpha^2 + 73939\alpha + 3159503 : 4790452)}$} &
\small{$\substack{1/532030685184(95326862066056021324942131004607377982417\alpha^3 \\ +
    545528649047435090745590753332103525334768\alpha^2 +\\
    5219097010417755802150885707010029194059739\alpha +
    26435647875049957400594991340415017745868228)}$}
& NO  \\
{$\substack{(2/3(-350\alpha^3 + 2133\alpha^2 - 20414\alpha + 131049) : 99\alpha^3 -
    821\alpha^2 + 7209\alpha - 10283 : \\ -121\alpha^3 + 765\alpha^2 - 4519\alpha +
    49527 : 3(\alpha^3 - 95\alpha^2 + 463\alpha + 7483) : 38628)}$} &
\small{$\substack{1/12838933875301847924736(272821842067812043459150463\alpha^3 -
    1441756036037485113891638193\alpha^2 + \\ 13628990993078146017200509961\alpha -
    81884012218007583784790937123)}$}
& NO  \\
{$\substack{(2255\alpha^3 + 12903\alpha^2 + 122561\alpha + 667341 : 727\alpha^3
    + 3783\alpha^2 + 38017\alpha + 283941 : \\ 205\alpha^3 + 1173\alpha^2 +
    13483\alpha + 128559 : 31\alpha^3 + 303\alpha^2 + 5305\alpha + 93933 :
    103008)}$} &
\small{$\substack{1/25657344(27395558102194290407\alpha^3 + 156777348815983490103\alpha^2 + \\
    1499898372357530447969\alpha + 7597242259694917251093)}$}
& NO  \\
    \end{tabularx}
\end{table}

 \begin{table}[h!] 
 \centering     
 \begin{tabularx}{0.82\paperheight}{>{\centering\arraybackslash}m{0.37\paperheight} >{\centering\arraybackslash}m{0.37\paperheight} >{\centering\arraybackslash}m{0.06\paperheight} }
 \multicolumn{3}{l}{Points with $\alpha$ satisfying $\alpha^4 -\alpha^3 - 26\alpha^2 - 612\alpha - 376 =0$.} \\ [0.8ex]
 Coordinates & $j$-invariant & CM   \\ [0.5ex]          \hline \hline \rule{0pt}{2.8ex}
{$\substack{(2(-7\alpha^3 + 819\alpha^2 + 338\alpha + 3140) : 2(\alpha^3 - 117\alpha^2 +
    4030\alpha + 7708) : \\ 41\alpha^3 - 39\alpha^2 - 1300\alpha + 11516 :
-43\alpha^3 + 273\alpha^2 + 2756\alpha + 11132 : 19032)}$} &
{$\substack{1/6030134083584(9696652518865066949806479791\alpha^3 +
    87629530896959309713118771172\alpha^2 + \\ 627432607595384179205760536627\alpha +
    363246817106763572205210032174)}$}
& NO  \\
{$\substack{(1/233064(-1727\alpha^3 + 2535\alpha^2 + 45058\alpha + 1115008) : 1/38844(-129\alpha^3
    + 130\alpha^2 + 3445\alpha + 104558) : \\ 1/621504(-953\alpha^3 + 1755\alpha^2 +
    24388\alpha + 953788) : 1/103584(-71\alpha^3 + 13\alpha^2 + 884\alpha + 116148) :
    1)}$} &
{$\substack{1/9984(-461824101632439297789343\alpha^3 + 753052741393911437371620\alpha^2 + \\
    11532547733418129305425109\alpha + 275363866633306279112305730)}$}
& NO  \\
{$\substack{(1/755872(1097\alpha^3 + 7553\alpha^2 + 26728\alpha + 269940) : 1/14536(9\alpha^3 +
    110\alpha^2 + 211\alpha + 7982) : \\ 1/377936(283\alpha^3 - 377\alpha^2 + 780\alpha +
    246548) : 1/188968(171\alpha^3 + 273\alpha^2 - 8710\alpha + 89880) : 1)}$} &
{$\substack{1/113724(4256408352336373807957394927\alpha^3 +
    38465549476528877862288479220\alpha^2 + \\ 275415602064898773665089788899\alpha +
    159449540270119811148234532430)}$}
& NO  \\
{$\substack{(4(-22207\alpha^3 + 29081\alpha^2 + 576784\alpha + 14718564) :
4(-10641\alpha^3 + 10907\alpha^2 + 310180\alpha + 7975676) :
\\ -21479\alpha^3 + 48581\alpha^2 + 609596\alpha + 17635140 :
2(-4957\alpha^3 + 4823\alpha^2 + 91364\alpha + 6258412) : 10977824)}$} &
{$\substack{1/1156284597843548784(-7533292230825378659055246790818024319\alpha^3 +\\
    12283824828743039994402351421616561892\alpha^2 + 
    188119355258177801813217995229096721397\alpha \\ +
    4491745817999145473916690222308277724034)}$}
& NO \\
    \end{tabularx}
\end{table}
\end{landscape}
\clearpage%
}

\newpage

\section*{Funding}

This work was supported by Deutsche Forschungsgemeinschaft
[DFG-Grant MU 4110/1-1 to S.G.]; and the Engineering and Physical Sciences Research Council [EP/N509796/1 to J.B., EP/N509619/1  to P.G.].

\section*{Acknowledgements}
The authors would like to warmly thank David Zureick-Brown for initiating this project, and both Jackson Morrow and David for their help during the project sessions of the 2020 Arizona Winter School. We are also grateful to the organisers of the school for providing such excellent conditions in Tucson, and the first and last named authors thank the school's donors for their financial support. 

We thank Samir Siksek for many valuable comments and suggestions, which have improved this work greatly. Likewise, we thank Tim Dokchitser, Tom Fisher, Steffen M\"uller, Filip Najman, Lazar Radi\v{c}evi\'c and Damiano Testa for their help. 

We are grateful to anonymous referees for many helpful comments that notably improved our article.
\bibliographystyle{alpha}
\bibliography{bib}
\end{document}